\documentclass[10pt]{amsart}
\usepackage[utf8]{inputenc}
\usepackage[english]{babel}
 
\usepackage{amssymb,amsmath,mathrsfs}  
\usepackage{tikz}             
\usepackage{multirow}        
\usepackage{rotating}
\usetikzlibrary{arrows}
\usepackage{verbatim}  
\usepackage{scrextend}
\usepackage{mathtools}
\usepackage{longtable}
\usepackage{centernot}
\usepackage{caption}
\usepackage{enumerate}

\DeclareMathOperator{\sgn}{sgn}

\def\XXint#1#2#3{{\setbox0=\hbox{$#1{#2#3}{\int}$}
     \vcenter{\hbox{$#2#3$}}\kern-.5\wd0}}

\def\rddots{\cdot^{\cdot^{\cdot}}}

\newtheorem{proposition}{Proposition}[section]             
\newtheorem{lemma}[proposition]{Lemma}
\newtheorem{theorem}[proposition]{Theorem}
\newtheorem{corollary}[proposition]{Corollary}
\newtheorem{remark}[proposition]{Remark}

\newtheorem{example}[proposition]{Example}
\newtheorem{definition}{Definition}[section]

\newcommand{\reals}{\mathbb{R}}

\newcommand{\A}{\mathcal{A}}
\newcommand{\cB}{\mathcal{B}}
\newcommand{\B}{\mathcal{B}}
\newcommand{\C}{\mathcal{C}}

\newcommand{\N}{\mathcal{N}}
\newcommand{\HH}{\mathcal{H}}
\newcommand{\J}{\mathcal{J}}
\newcommand{\cP}{\mathcal{P}}
\newcommand{\Q}{\mathcal{Q}}
\newcommand{\R}{\mathcal{R}}
\newcommand{\cS}{\mathcal{S}}
\newcommand{\T}{\mathcal{T}}
\newcommand{\W}{\mathcal{W}}
\newcommand{\Y}{\mathcal{Y}}

\newcommand{\ccC}{\mathscr{C}}
\newcommand{\oast}{\circledast}

    \makeatletter													   
    \renewcommand\section{\@startsection{section}{1}{\z@}               
                                      {-3.5ex \@plus -1ex \@minus -.2ex}%
                                      {2.3ex \@plus.2ex}                %
                                      {\normalfont\bfseries}}           %
    \makeatother                                                        %

\begin{document}


\title{Refined Inertia of Matrix Patterns}
\thanks{Research supported in part by an NSERC USRA (Earl), as well as NSERC Discovery Grants 203336 (Vander Meulen) 
and 2014-03898 (Van Tuyl).}

\author{Jonathan Earl}
\address{Department of Mathematics and Statistics,
McMaster University, Hamilton, ON, L8S 4L8, Canada}
\email{earlj@math.mcmaster.ca}

\author{Kevin Vander Meulen}
\address{Department of Mathematics, Redeemer University College, Ancaster, ON, L9K 1J4, Canada}
\email{kvanderm@redeemer.ca}

\date{\today}

\author{Adam Van Tuyl}
\address{Department of Mathematics and Statistics,
McMaster University, Hamilton, ON, L8S 4L8, Canada}
\email{vantuyl@math.mcmaster.ca}


\begin{abstract}
We explore how the 
combinatorial arrangement of prescribed zeros in a matrix
affects the possible eigenvalues that the matrix can obtain.
We demonstrate that there are inertially arbitrary patterns 
having a digraph with no $2$-cycle, unlike what happens for nonzero patterns.
We develop a class of patterns that are refined inertially arbitrary
but not spectrally arbitrary, making use of the property of a properly signed nest. 
We include a characterization of the
inertially arbitrary and refined inertially arbitrary patterns of order three,
as well as the patterns of order four with the least number of nonzero entries.
\end{abstract}

\maketitle

\section{Introduction}

Since the concepts were introduced in Drew et. al.~\cite{Drew}, much work has focused on when
a sign pattern is spectrally or inertially arbitrary (for example \cite{CF, CVV, GS, KOV}). 
Various papers have also focused on the combinatorial arrangement of the nonzero
positions in a matrix when considering eigenvalue properties of patterns (for example \cite{CV4, CM, DOV}).
The paper by Cavers and Fallat~\cite{CF} reviews some of these results and also sets them
in a more general setting, initiating the study of other types of patterns. Recently, there has started to be some
consideration of zero-nonzero patterns, patterns with prescribed zero entries, nonzero entries and
entries that are unrestricted (see for example, \cite{ESV}). The concept of a refined inertially arbitrary
pattern has also recently been defined in \cite{DOV} in the context of nonzero patterns. This concept
includes patterns that allow for Hopf bifurcations \cite{BDMOV}.
In this paper, we begin an exploration of zero patterns (patterns that have prescribed zero entries, 
with the remaining entries unrestricted) that are refined inertially arbitrary, comparing them to 
other known results.
 
We first introduce the technical definitions and some lemmas in Section~\ref{defines}. We recall that a refined inertially arbitrary zero pattern requires a pair of symmetrically opposite nonzero entries, as was observed in~\cite{CF}. We then show, 
in Section~\ref{withou}, that this restriction is not required for inertially arbitrary zero patterns, unlike what happens for nonzero patterns \cite{CV4} and sign patterns \cite{CV}.

In \cite{DOV}, Deaett et. al. present an irreducible nonzero pattern of order $5$ that is refined inertially arbitrary but not
spectrally arbitrary. In Section~\ref{pathsection}, we produce an infinite class of zero patterns that are refined inertially arbitrary
but not spectrally arbitrary. 

Garnett and Shader~\cite{GS} showed that a path (sign) pattern $\T_n$ having two end loops is spectrally arbitrary. 
In Section~\ref{supersection}, using the technique from \cite{GS}, we show that there are other zero patterns corresponding to a path with two loops that are spectrally arbitrary. While $\T_n$ has a signing that is spectrally arbitrary, there is no signing of the nonzero entries of the path pattern introduced in Section~\ref{supersection} that is spectrally arbitrary.

In Section~\ref{order3}, we characterize the zero patterns of order $3$ that are refined inertially arbitrary and inertially
arbitrary, comparing them to the known spectrally arbitrary patterns. 
We say a pattern of order $n$ is \emph{sparse} if it has less than $2n$ nonzero entries.
In Section~\ref{order4} we conclude with an exploration of the sparse zero patterns of order $4$ 
determining which are refined inertially arbitrary, spectrally arbitrary, or inertially arbitrary.

\section{Technical definitions and lemmas for inertially arbitrary patterns}\label{defines} 
A \emph{sign pattern} is an order $n$ matrix with entries in $\{+,-,0\}$; a \emph{nonzero pattern} has entries in $\{ *,0 \}$. 
The \emph{qualitative class} of a sign pattern $\A$, denoted $Q(\A)$, is the set of all real matrices $A$
such that $\sgn{(A_{ij})}=\sgn{(\A_{ij})}$. Likewise the qualitative class of a nonzero 
pattern $\A$ consists of all the real matrices $A$ with $A_{ij}$ nonzero if and only if $\A_{ij}\neq 0$. 
A \emph{zero pattern} $\A$  is a matrix with entries in $\{\oast, 0\}$; $Q(\A)$ is the set of real matrices
$A$ such that $\A_{ij}=0$ implies $A_{ij}=0$. We refer to the $\oast$ entries of $\A$ as the nonzero entries
of the pattern, even though a matrix $A\in Q(\A)$ may have a zero in that position.

A pattern $\A$ \emph{realizes a polynomial} $p(x)$ if there exists a matrix $A\in Q(\A)$ such that the characteristic polynomial of $A$ is $p(x)$, and $\A$ is \emph{spectrally arbitrary} if $\A$ realizes all monic polynomials of degree $n$ with real coefficients. The \emph{inertia} of an order $n$ matrix $A$, denoted $i(A)$, is the ordered tuple $i(A) = (n_+, n_-, n_0)$ where $n_+$ (resp. $n_-, n_0$) is the number of eigenvalues of $A$ with positive real part (resp. negative and zero). 
The \emph{refined inertia} of a matrix $A$, $ri(A) = (n_+, n_-, n_z, n_i)$ includes 
$n_i$, the number of the eigenvalues of $A$ that are purely imaginary,  and $n_z$, the number of eigenvalues of $A$ that are zero.
Note that $n_z + n_i = n_0$. $\A$ can \textit{realize} an inertia (resp. refined inertia) $\mathbf{a}$ if there exists an $A\in Q(\A)$ such that $i(A) = \mathbf{a}$ (resp. $ri(A) = \mathbf{a}$). A pattern $\A$ is \emph{inertially arbitrary}  if $\A$ can realize all inertias $(n_1, n_2, n_3)$ with $n_1 + n_2 + n_3 = n$. Likewise, $\A$ is \textit{refined inertially arbitrary}  if $\A$ can realize all refined inertias $(n_1, n_2, n_3, n_4)$ where $n_1 + n_2 + n_3 + n_4 = n$. 
Note that for nonzero patterns and zero patterns,  
if $\A$ can realize the inertia $(n_+, n_-, n_0)$, then the \emph{reversal} $(n_-, n_+, n_0)$ can also be realized by $\A$ by taking the negative of the matrix used to realize $(n_+, n_-, n_0)$. Since we will be focusing on zero patterns, when considering inertias, we will restrict to inertias with $n_+ \geq n_-$.

A matrix $A$ (or pattern) of order $n$ is \emph{irreducible} if there does not exist a permutation matrix $P$, such that $PAP^{-1}$ is a block triangular matrix, with two or more non-empty diagonal blocks. We will focus on irreducible patterns since the
eigenvalues of a reducible matrix can be obtained from its irreducible blocks. As noted in the concluding comments, it may
be worth considering reducible patterns in future work.

 
Note that any pattern $\A$ of order $n$ can be represented by a corresponding digraph $D$ on $n$ vertices: 
$\A$ is the adjacency matrix of the digraph $D=D(\A)$ with vertex set $\{v_i | 1 \leq i \leq n\}$ and arc set $\{(v_i, v_j) | \A_{i, j} \neq 0\}$. Note that $\A$ is irreducible if and only if $D(\A)$ is strongly connected (see e.g. \cite[Theorem 3.2.1]{B}).
Furthermore, if $\A$ is a pattern and $\A^T$ is the transpose of $\A$ with digraphs $D$ and $D^T$, respectively, $D^T$ is obtained from $D$ by reversing all the arcs. Two patterns $\A$ and $\cB$ are \emph{equivalent} if $\cB$ can be obtained from $\A$ via permutation similarity and/or transposition. 

If $D(\A)$ has a cycle 
$(v_{i_1},v_{i_2}),(v_{i_2},v_{i_3}),\ldots,(v_{i_{k}},v_{i_1})$
and $A=[a_{ij}]\in Q(\A),$ then the associated \emph{cycle product} is $(-1)^{k-1}a_{v_{i_1},v_{i_2}}a_{v_{i_2},v_{i_3}}\cdots a_{v_{i_{k}},v_{i_1}}$. A \emph{composite} cycle of length $k$ is a set of vertex disjoint cycles with
lengths summing to $k$ (with an associated cycle product being the product of the cycle products of the individual
cycles). Note that a composite cycle $C$ could consist of only a single cycle; in this case, we sometimes
call $C$ a \emph{proper} cycle.
A $1$-cycle is an
arc from a vertex back to itself; this arc is often called a \emph{loop}.
We will use the following fact (see for example
~\cite[Section 9]{B}):
\begin{lemma}\label{cycleprod} 
Given $A\in Q(\A)$, if $E_k$ is the sum of all the composite cycle products of length $k$ in $D(\A)$,
then the characteristic polynomial of  
$A$ is 
\begin{eqnarray*}
p_A(x)=x^n-E_1x^{n-1}+E_2x^{n-2}+\cdots+(-1)^nE_n.
\end{eqnarray*}
\end{lemma}

\begin{example}\label{Cn}
{\rm 
The pattern
$$\C_n = \left[ \begin{array}{ccccc}
\oast 	& \oast 	& 0			& \cdots 	& 0\\
\oast 	& 0     	& \oast		& \ddots  	& \vdots\\
\oast 	& 0     	& 0       	& \ddots	& 0\\
\vdots	& \vdots	& \ddots  	& \ddots 	& \oast \\
\oast	& 0		    & \cdots  	& 0  		& 0 \end{array} \right]$$
 is known to be spectrally arbitrary for all $n\geq 2$, since it is the pattern of a companion matrix
 (see, e.g., \cite{CF}).
 One could use a matrix realization with ones on the superdiagonal, along with Lemma~\ref{cycleprod}
 to verify $\C_n$ can obtain every characteristic polynomial of degree $n$.
}\end{example} 
  We use the following lemmas to develop Theorem~\ref{RIAP3}, a main result of Section~\ref{order3}.
The first lemma was observed in~\cite[Lemma 20]{KOV}. Further necessary conditions on the coefficients
of a characteristic polynomial based on inertia can be found in~\cite[Lemma 1]{CVV}.

\begin{lemma}\cite[Lemma 20]{KOV}\label{kcycle}
If a pattern $\A$ of order $n$ is inertially arbitrary, then $D(\A)$ must have a composite
$k$-cycle for each $k$, $1\leq k\leq n$. 
\end{lemma}

\begin{proof}
It is enough to observe that if $A$ has inertia $(0,n,0)$, then the characteristic polynomial
of $A$ has all positive coefficients. Thus the result follows from 
Lemma~\ref{cycleprod}.
\end{proof}

The next lemma indicates that the digraph of refined inertially arbitrary 
patterns require a proper $2$-cycle, not merely a composite $2$-cycle, 
as observed by Cavers and Fallat in \cite[Cor. 2.8]{CF}). 
While Cavers and Fallat used 
the refined inertia $(0,0,a,b)$ with
$b=2$ in their hypothesis, we observe that
for any $b\geq 2$, if $A\in\A$ has refined inertia $(0,0,a,b),$ then $D(\A)$ 
has a $2$-cycle. We include the
argument for completeness.

\begin{lemma}\label{2cycle}
If a pattern $\A$ allows refined inertia $(0,0,a,b)$ with $b\geq 2$, then $D(\A)$ has a proper $2$-cycle.
\end{lemma}

\begin{proof} 
Suppose $A\in \A$ has refined inertia $(0,0,n-2c,2c)$ for some $c\geq 1$.
In particular, suppose $A$ has nonzero eigenvalues $\pm \ell_1i,\pm \ell_2i,\ldots,\pm \ell_ci$.
Note that the trace of $A$ is zero since it is the sum of the eigenvalues. If 
$p_A(x)=x^n-E_1x^{n-1}+E_2x^{n-2}+\cdots+(-1)^nE_n$, then
since  $E_2$ is the sum of all products of two eigenvalues, $E_2=\sum\limits^c_{k=1} \ell_k^2 >0$.
By Lemma~\ref{cycleprod},  $E_2$ is also the sum of the signed composite $2$-cycles, so
$E_2=\sum\limits_{k<j} a_{kk}a_{jj} - \sum\limits_{k<j} a_{kj}a_{jk}$. 
If $D(\A)$ has no proper $2$-cycles, then 
$$2E_2=2\sum\limits_{k<j} a_{kk}a_{jj}=({\rm{tr}}(A))^2-\sum a_{kk}^2=-\sum a_{kk}^2\leq 0.$$
But this would contradict the fact that $E_2>0$ as noted earlier. Thus $D(\A)$ must have
a proper $2$-cycle.
\end{proof}

In the next lemma, we note that if the digraph of a pattern has $n$ loops, then the pattern is inertially arbitrary. A pattern $\B$
is a \emph{superpattern} of a pattern $\A$ if $\B_{ij}=0$ implies $\A_{ij}=0$. A superpattern $\B$ is a
\emph{proper} superpattern of $\A$ if $\B\neq \A$. Note that if $\A$ is allows an eigenvalue property, then any superpattern
of $\A$ will also allow that property. The following lemma is an example for the property of
being inertially arbitrary.

\begin{lemma}\label{iapdiagonal}
If $\A$ is 
pattern of order $n$ with $\oast$ entries in all its diagonal positions, then $\A$ is inertially arbitrary.
\end{lemma}

\begin{proof}
In this case, $\A$ is a superpattern of the reducible diagonal pattern which is inertially arbitrary.
\end{proof}

\section{An inertially arbitrary pattern without a $2$-cycle}\label{withou}

Unlike for nonzero patterns~\cite[Lemma 2.1]{CV4}, and sign patterns~\cite[Lemma 5.1]{CV}, the digraph of an irreducible inertially arbitrary zero pattern does not require a proper $2$-cycle, as we will see in Theorem~\ref{IAPN}. 
For $n\geq 3$, let
$$\A_n = \left[ \begin{array}{ccccc}
\oast 	& \oast 	& 0			& \cdots 	& 0\\
0 		& \ddots	& \ddots		& \ddots  	& \vdots\\
\vdots 	& \ddots	& \oast  	& \oast		& 0\\
0	& \ddots	& \ddots  	& \oast 		& \oast \\
\oast  	& 0		& \cdots  	& 0  		& 0 \end{array} \right].$$

\begin{theorem}\label{IAPN}
	For $n\geq 3$, the pattern $\A_n$ is inertially arbitrary but
not refined inertially arbitrary.
\end{theorem}

\begin{proof} By Lemma~\ref{2cycle}, $\A_n$ is not refined inertially arbitrary.
	Let 
	$$A=A(a_1,a_2,\ldots,a_{n-1},c) = \left[ \begin{array}{ccccc}
a_1 		& 1 		& 0			& \cdots 	& 0\\
0 		& a_2	& 1			& \ddots  	& \vdots\\
\vdots 	& \ddots	& \ddots  	& \ddots		& 0\\
0	& \ddots	& \ddots  	& a_{n-1} 		& 1 \\
c	  	& 0		& \cdots  	& 0  		& 0 \end{array} \right].$$ 
Using Lemma~\ref{cycleprod}, we see that the characteristic polynomial of $A$ is $p_A(x)=f(x)-c$. In fact,
 $f(x)$ does not depend on $c$ since $c$ only appears as a weight on 
the $n$-cycle of $D(\A_n)$. Further, $f(0)=0$ since $\det(A)=(-1)^{n-1}c$.
Note that $A(a_1,a_2,\ldots,a_{n-1},0)$ is upper triangular and hence has 
eigenvalues $a_1,a_2,\ldots,a_{n-1}$ and $0$. Thus $\A_n$ can realize any
inertia $(n_+,n_-,n_0)$ with $n_0\geq 1$.

To obtain the remaining inertias, we use the idea that if you take a polynomial function with
distinct roots, with one root at zero, and shift it
down slightly, then the resulting function will have
distinct nonzero roots. The slope of the original
polynomial at origin will determine if you gained a positive or
negative root. In particular, let $A(c)=A(d_1,d_2,\ldots,d_{n-1},c)$ be a realization of $\A$ with
distinct $d_k$, $1\leq k\leq n-1$, such that $i(A(0))=(a,b,1)$ for some $a,b \geq 0$. Then
$p_{A(0)}(x)=f(x)$ and $f'(0)\neq 0$ since 0 is not a double root of $f(x)$.
In this case, without loss of generality, assume $f'(0)>0$. Choose 
$\epsilon>0$ such that $\epsilon < {\rm{min}}\{|f(x)| : f'(x)=0 \}$. Since
$p_{A(\epsilon)}(x)=f(x) +\epsilon$, we have $i(A(\epsilon))=(a+1,b,0)$
and $i(A(-\epsilon))=(a,b+1,0)$. Thus $\A_n$ can realize any
inertia $(n_+,n_-,n_0)$ with $n_0 = 0$.
\end{proof}

\begin{remark}{\rm
The pattern in Theorem~\ref{IAPN} is an order $n$ example of a zero pattern that is inertially but not refined inertially
arbitrary. In~\cite{CF}, Cavers and Fallat presented an order $4$ zero pattern $\N$ with this property.\footnote{Other such patterns can be found in \cite{CVV}, or 
[Kim, In-Jae, D. D. Olesky, and P. van den Driessche. ``Inertially arbitrary sign patterns with no nilpotent realization.''
\emph{Linear Alg. Appl.} 421.2 (2007) 264--283].  Our new examples have less than $2n$ nonzero entries.}
In Section~\ref{order4}, we provide some other order $4$ zero patterns with this property.
}\end{remark}

\begin{remark} {\rm
The inertially arbitrary pattern in Theorem~\ref{IAPN} is \emph{minimal} in that if any of the $\oast$ entries are replaced by $0$, then the resulting pattern would no longer be an irreducible inertially arbitrary pattern. 
In particular: the $\oast$ entries corresponding to the $n$-cycle must remain nonzero otherwise the pattern is not irreducible; and replacing any $\oast$ on the diagonal 
with zero 
would result in a pattern that has no composite $(n-1)$-cycle and hence could not be
inertially arbitrary by Lemma~\ref{kcycle}.
}\end{remark}

\section{Refined inertially arbitrary path patterns}\label{pathsection}
In this section we explore irreducible patterns whose underlying graph is a path. In this case, the largest cycle in the
graph is a $2$-cycle. In particular, the digraph of such a pattern must 
include $(n-1)$ $2$-cycles (for irreducibility) and at least one loop (by Lemma~\ref{kcycle}). We characterize which
of these patterns with $2n-1$ nonzero entries are refined inertially arbitrary. 

Given $1\leq \alpha \leq n$, let 
$$\cP_{n, \alpha} = \left[ \begin{array}{ccccc}
0    	& \oast & 0			& \cdots  	& 0\\
\oast 	& 0 & \oast		& \cdots  	& \vdots\\
0 		& \oast	& \oast  	&\ddots  	&0\\
\vdots	& 0		& \ddots  	& 0 		&\oast\\
0   	&\cdots &0  	    & \oast  	&0\end{array} \right]$$ 
be the order $n$ pattern with $2n-1$ nonzero entries such that the nonzero diagonal element is in row $\alpha$. 
Note that $\cP_{n,\alpha}$ is not spectrally arbitrary since it does not allow a characteristic polynomial with
the coefficient of $x^{n-1}$ zero while the coefficient of $x^{n-3}$ is nonzero. 
For example, there is no matrix $A\in Q(\cP_{n,\alpha})$ with characteristic polynomial $x^n+x^{n-3}$.

We use the following concept: 
\begin{definition}
{\rm Suppose $B$ is an order $n$ real matrix. Let $\alpha_1, \alpha_2,\ldots,\alpha_n$ be a rearrangement of the elements of $\{1,\ldots,n\}$, and $B\left[\{\alpha_1 ,\ldots, \alpha_k\}\right]$ denote the principal submatrix in rows and columns $\alpha_1 ,\ldots, \alpha_k$ of $B$. 
A sign pattern $\A$ \emph{allows a nested sequence of properly signed principal minors} (abbreviated to a \emph{properly signed nest}) if there exist $B \in Q(\A)$ and $\alpha_1,\ldots, \alpha_n$, such that $$\sgn(\det(B\left[\{\alpha_1,\ldots, \alpha_k\}\right])) = (-1)^k \text{ for } k=1,\ldots,n.$$ In this case, we refer to $\left[\{\alpha_1,\ldots, \alpha_n\}\right]$ as a properly signed nest of $B$.
}\end{definition}

\begin{example}\label{pathLemma1} {\rm Define $P_{n, \alpha}$ to be the matrix with pattern $\cP_{n, \alpha}$ having $1$'s on the subdiagonal and $-1$'s on the remaining $n$ nonzero positions. As noted in \cite[Example~2]{Maybee}, 
$P_{n, 1}$ has a properly signed nest, namely $[\{1,2,3,\ldots,n\}]$.
}
\end{example}

\begin{lemma}\cite[Theorem~2.1]{Johnson}\label{pathLemma0}
If $\A$ is an order $n$ sign pattern that allows a properly signed nest, then $\A$ allows inertias $(n, 0, 0)$ and $(0, n, 0)$.
\end{lemma}

\begin{theorem}\label{pathRIAP}
The pattern $\cP_{n, 1}$ is refined inertially arbitrary for all $n \geq 1$. 
\end{theorem}

\begin{proof}
Note that for $n \leq 2$,  $\cP_{n, 1} = \C_n$, the pattern in Example~\ref{Cn}. Hence, it is refined inertially arbitrary. For $n>2$, note that $$\cP_{n, 1} = \left[ \begin{array}{cc}
\cP_{n-2, 1} 	& \Q \\
\cS 			& \T \end{array} \right]$$ with $\cS = \left[ \begin{array}{cccc}
0   & \cdots & 0 &\oast \\
0   & \cdots & 0 & 0 \end{array} \right]$, $\Q = \cS^T$, and $\T = \left[ \begin{array}{cc}
0   & \oast \\
\oast & 0 \end{array} \right]$.
	Note that the matrices $\left[ \begin{array}{cc}
0 & 1 \\
1 & 0 \end{array} \right]$, $\left[ \begin{array}{cc}
0 & 0 \\
1 & 0 \end{array} \right]$, and $\left[ \begin{array}{cc}
0 & -1 \\
1 & 0 \end{array} \right]$ all have the pattern $\T$, so $\T$ can realize the refined inertias $W = \{(1, 1, 0, 0), (0, 0, 2, 0), (0, 0, 0, 2)\}$. $\T$ does not allow the refined inertias $(2, 0, 0, 0)$ or $(1, 0, 1, 0)$. Let $\cP'_{n, 1}$ be the reducible subpattern of $\cP_{n, 1}$ with $\cS=0$. 
	By induction, $\cP'_{n, 1}$ (and hence $\cP_{n, 1}$) can realize any refined inertia of the form $(a, b, c, 2d) + (x, y, z, w)$, such that $a+b+c+2d=n-2$ and $(x, y, z, w) \in W$. Thus, $\cP'_{n, 1}$ can realize all refined inertias of the form $(\alpha, \beta, \gamma, 2\delta)$, where $\alpha + \beta + \gamma + 2\delta = n$, and either $\alpha$, $\beta \geq 1$, or $\gamma \geq 2$, or $\delta \geq 1$.
	This means that we have shown that $\cP'_{n, 1}$ allows every refined inertia except possibly ($n, 0, 0, 0$) and $(n-1, 0, 1, 0$), up to reversal.
	
	By Example \ref{pathLemma1} and Lemma \ref{pathLemma0}, $\cP_{n, 1}$ allows refined inertia $(0, n, 0, 0)$ and its reversal $(n, 0, 0, 0)$. Consider the reducible subpattern of $\cP_{n, 1}$ with entry $(n-1, n)$ set to zero. 
	Lemma \ref{pathLemma0} implies that $\cP_{n, 1}$ allows inertias $(n-1, 0, 1, 0)$ and $(0, n-1, 1, 0)$.
		Thus, by induction $\cP_{n, 1}$ is refined inertially arbitrary.
	\end{proof}

\begin{lemma}\label{badPos}
$\cP_{n, \alpha}$ does not allow a properly signed nest if $n$ is odd, and $\alpha$ is even.
\end{lemma}

\begin{proof}
Suppose $n$ is odd, and $\alpha$ is even. Then, $\cP_{n, \alpha}$ has no nonzero transversal, that is, 
$D(\cP_{n, \alpha})$ has no composite $n$-cycle. In this case, $\det(P) = 0$
for each $P\in \cP_{n, \alpha}$. Hence, $\cP_{n, \alpha}$ does not allow a properly signed nest.
\end{proof}

\begin{remark}\label{R}{\rm
Note that if 
\begin{equation*}
R = \left[ \begin{array}{ccc}
0 	&       & 1	\\
    &\rddots& \\
1 	&       & 0 \end{array} \right] {\rm{\quad and \quad}}
\mathbf{e} = \left[0,\ldots,0, 1\right]
{\rm{ \qquad then \quad}}
P_{n, n-1} = \left[ \begin{array}{cc}
RP_{n-1, 1}R 		& -\mathbf{e}^T	\\
\mathbf{e} 		& 0 \end{array} \right]. 
\end{equation*}
 Additionally, $\cP_{n, \alpha}$ is equivalent to $\cP_{n, n-\alpha +1}$, since $\cP_{n, \alpha} = R\cP_{n, n-\alpha +1}R.$
}\end{remark}

\begin{lemma}\label{pathLemma2}
$P=P_{n, n-1}$ has a properly signed nest if and only if $n$ is even. Further, $[\{n-1,n-2,\ldots,1,n\}]$
is a properly signed nest of $P$ if $n$ is even.
\end{lemma}

\begin{proof}
By Lemma~\ref{badPos}, if $\cP_{n, n-1}$ allows a properly signed nest, then $n$ must be even. Note that $\cP_{2, 1}$ allows a properly signed nest by Example~\ref{pathLemma1}. Suppose that $n$ is even, and $n > 2$. 
The matrix $P_{n-1, 1}$ has a properly signed nest as noted in Example~\ref{pathLemma1}, 
so $\det(P_{n-1, 1}) < 0$, since $n-1$ is odd. Thus $\det(RP_{n-1}R)<0$ 
and $RP_{n-1}R$ has the properly 
signed nest $[\{n-1,n-2,\ldots,2,1\}]$.  By Remark~\ref{R}, 
it is enough to show that $\det(P_{n, n-1}) > 0$. By cofactor expansion on the first column, 
$$\det(P_{n, n-1}) = -\det\left(\left[ \begin{array}{c|ccccc}
-1 		& 0			& 0		&0		& \cdots 	& 0\\ \hline
1 		& 0			& -1		& \ddots	& \ddots  	& \vdots\\
0	 	& 1			& \ddots	& \ddots	& \ddots		& 0\\
0	 	& \ddots		& \ddots	& 0		& -1			& 0\\
\vdots	& \ddots		& \ddots	& 1		& -1 		& -1 \\
0	  	& \cdots		& 0		& 0		& 1  		& 0 \end{array} \right]\right) = \det(P_{n-2, n-3}).$$
It follows by induction that $\det(P_{n, n-1}) > 0$. Hence, $P_{n, n-1}$ has a the 
properly signed nest $[\{n-1,n-2,\ldots,1,n\}].$  
\end{proof}

\begin{lemma}\label{pathLemma3}
Let $\cP = \cP_{n, \alpha}$. Then $\cP$ allows a properly signed nest  if and only if $n$ is even or $\alpha$ is odd.
In this case, $\left[\{\alpha, \alpha-1, \alpha-2,\ldots,1, \alpha+1, \alpha+2,\ldots,n\}\right]$ is a properly signed nest of $P_{n,\alpha}$.
\end{lemma}

\begin{proof}
By Lemma~\ref{badPos}, $\alpha$ must be odd if $n$ is odd. Suppose that either $n$ is even, or $n$ is odd and $\alpha$ is odd. Observe that if $n$ is even and $\alpha$ is even, $\cP_{n, \alpha}$ is equivalent to $\cP_{n, n-\alpha+1}$, and $n-\alpha+1$ is odd. Thus, we may assume that $\alpha$ is odd.
The case when $\alpha=1$ is covered by Theorem~\ref{pathRIAP} and Example~\ref{pathLemma1}. The case when
$\alpha=n-1$ is covered by Lemma~\ref{pathLemma2}. 
 
Now assume  
that $1 < \alpha < n-1$. 
We claim that $\left[\{\alpha, \alpha-1, \alpha-2,\ldots,1, \alpha+1, \alpha+2,\ldots,n\}\right]$ is a properly signed nest for $P_{n, \alpha}$. 
By induction, it is enough to show that $\sgn(\det(P_{n, \alpha})) = - \sgn(\det(P_{n-1, \alpha}))$. 
By cofactor expansion along the last column of $P_{n, \alpha}$, 
$\det(P_{n, \alpha}) = \det\left(\left[ \begin{array}{cc}
P_{n-2, \alpha} 	& * \\
0 		& 1 \end{array} \right] \right) = \det(P_{n-2, \alpha}).$ 
Since $P_{n-1, \alpha}$ has a properly signed nest, $\sgn(\det(P_{n-2, \alpha})) = \sgn(\det(P_{n, \alpha})) 
= -\sgn(\det(P_{n-1, \alpha}))$. Thus, $P_{n, \alpha}$ has a properly signed nest, and $\left[\{\alpha, \alpha-1, \alpha-2,\ldots,1, \alpha+1, \alpha+2, \ldots,n\}\right]$ is a properly signed nest for $P_{n, \alpha}$.
\end{proof}

\begin{theorem}\label{patheorem} Suppose $n\geq 1$. Then 
$\cP_{n, \alpha}$ is refined inertially arbitrary if and only if $n$ is even or $\alpha$ is odd.
\end{theorem}

\begin{proof}
For both $n = 1$ and $n = 2$, $\cP_{n, \alpha} = \C_n$, which are both spectrally arbitrary. 
Let $n = k$ for some $k\geq 3$ and assume that the claim is true for all patterns
$\cP_{m,\alpha}$ with $m\leq k-1$ and $1\leq \alpha \leq m$.
 First note that we can assume that $\alpha >1$ by Theorem~\ref{pathRIAP}.
Further,  we can assume that $\alpha < n-1$ since $\cP_{n, \alpha}$ is equivalent to $\cP_{n, n-\alpha+1}$
 as noted in Remark~\ref{R}. Note, $\cP_{n, \alpha} = \left[ \begin{array}{cc}
\cP_{n-2, \alpha} 	& \Q \\
\cS 			    & \T \end{array} \right]$, with $\Q, \cS$, and $\T$ as defined in the proof of 
Theorem~\ref{pathRIAP}. As in the proof of Theorem~\ref{pathRIAP}, 
$\T$ can realize the refined inertias $(1, 1, 0, 0)$, $(0, 0, 2, 0)$, and $(0, 0, 0, 2)$. 
Replacing $\Q$ with a zero matrix, the resulting block triangular pattern can be used to inductively show that
$\cP_{n,\alpha}$ allows all refined inertias, except possibly $(n, 0, 0, 0)$ and $(n-1, 0, 1, 0)$ and their reversals. 
 By Lemmas~\ref{pathLemma0} and \ref{pathLemma3}, $\cP_{n, \alpha}$ allows $(n, 0, 0, 0)$. 
If $\alpha$ is odd, we can set entry $(n,n-1)$ to be zero to obtain a block triangular
subpattern of $P_{n,\alpha}$ with blocks $P_{n-1,\alpha}$ and $[0]$. Hence inductively, $P_{n,\alpha}$ 
allows refined inertia $(n-1, 0, 1, 0)$. If $n$ is even and $\alpha$ is even, setting
entry $(2,1)$ to be zero gives a block triangular pattern with blocks $[0]$ and $\cP_{n-1,\alpha -1}$
which inductively allows refined inertia $(n-1,0,1,0)$.
 Thus, $\cP_{n, \alpha}$ is refined inertially arbitrary.
\end{proof}

It may be interesting to characterize which superpatterns of $\cP_{n,\alpha}$ are refined inertially arbitrary.
When $n$ is odd and $\alpha$ is even, then adding more loops to even positions, and only even positions, will not produce 
a refined inertially arbitrary pattern since such a pattern has zero determinant, as in Lemma~\ref{badPos}.
In the next section, we explore some superpatterns that are spectrally arbitrary, and hence, refined inertially arbitrary. 

\section{Superpatterns of $\cP_{n,\alpha}$.}\label{supersection}

In the previous section we demonstrated that while the path patterns with one loop are not
spectrally arbitrary, some are refined inertially arbitrary. 
Inserting exactly one additional $\oast$ in some $\cP_{n, \alpha}$ patterns can produce a spectrally arbitrary
superpattern, as will be demonstrated in this section. 

A useful technique is the nilpotent-centralizer method introduced in \cite{GS}.
While this technique was introduced for sign patterns, it 
also applies to the zero patterns we have been discussing, 
as noted in \cite{ESV} and described in the 
next theorem. A matrix $N$ is \emph{nilpotent} if $N^k=0$ for some positive integer $k$. A 
nilpotent matrix has \emph{index} $k$ if $k$ is the smallest positive integer such that $N^k=0$.
The notation $A\circ B$ represents the Hadamard product of $A$ and $B$. We use the fact that
an entry of $A\circ B$ is nonzero if and only if the corresponding entries of both $A$ and $B$ are nonzero. 

\begin{theorem}\cite{GS}\label{NC}
Let $A$ be an order $n$ nilpotent matrix of index $n$ with pattern $\A$. 
If the only matrix $B$ in the centralizer of $A$ satisfying $B^T \circ {\A} = 0$ is the zero matrix,
then  $\A$  is spectrally arbitrary.
\end{theorem}

For example, using  the nilpotent-centralizer method, Garnett and Shader~\cite{GS} demonstrated
 that $$\T_n = \left[ \begin{array}{ccccc}
\oast	& \oast 	& 0		    & \cdots& 0 \\
\oast 	& 0 		& \ddots 	& \ddots& \vdots \\
0 		& \ddots 	& \ddots 	&\oast	& 0 \\
\vdots 	& \ddots    & \oast 	& 0 		&\oast \\
0 		& \cdots	& 0 		&\oast	& \oast \end{array} \right]$$ is spectrally arbitrary for $n\geq 2$. 
For $n\geq 3$, let 
$$\W_{n} = \left[ \begin{array}{cccccc}
\oast	& \oast 	& 0		    & \cdots& \cdots    & 0 \\
\oast 	& 0 		& \ddots 	& \ddots&           & \vdots \\
0 		& \ddots 	& \ddots 	&\oast	& \ddots    & \vdots \\
\vdots 	& \ddots    & \oast 	& 0 	&\oast      & 0\\
\vdots 	&       	& \ddots 	&\oast	& \oast     &\oast \\
0       & \cdots    & \cdots    &0      & \oast     & 0\end{array} \right] = \left[ \begin{array}{cc}
\T_{n-1} 	& \mathbf{e}^T \\
\mathbf{e}	& 0 \end{array} \right]$$
with $\mathbf{e} = \left[0,\ldots,0, 1\right].$

\begin{theorem}
If $n\geq 3$, then $\W_{n}$ is a spectrally arbitrary pattern.
\end{theorem}

\begin{proof} Let $n\geq 3$ and
    $$M = \left[ \begin{array}{ccccc}
m_{1, 1}& m_{1, 2} 	& 0		    & \cdots    & 0 \\
m_{2, 1}& 0 		& \ddots 	& \ddots    & \vdots \\
0 		& \ddots 	& \ddots 	&m_{n-3, n-2}	& 0 \\
\vdots 	& \ddots    & m_{n-2, n-3} 	& 0 	    & m_{n-2, n-1} \\
0 		& \cdots	& 0 		&m_{n-1, n-2}	& m_{n-1, n-1} \end{array} \right]$$ be the nilpotent matrix used in~\cite[Corollary~8]{GS} to show that $\T_{n-1}$ is spectrally arbitrary.  This means that if $A$ is a matrix in the centralizer
of $M$ with $A^T\circ\T_{n-1}=0$, then $A=0$.
Note also that the $\oast$ entries of 
$\T_{n-1}$ are nonzero in $M$. 

Let $N = \left[ \begin{array}{cc}
M & \mathbf{e}^T \\
0 	& 0 \end{array} \right].$  Then $N$ 
is a nilpotent matrix with pattern $\W_{n}$. 
Since $N$ is a proper Hessenberg matrix,
$N$ has nilpotent index $n$ (see, for example \cite[Theorem 7.4.4]{GV}).

Suppose $B$ is a matrix such that
$B^T\circ \W_{n}=0$.
 Then $B = \left[ \begin{array}{cc}
A     & \mathbf{y} \\
\mathbf{x} & a \end{array} \right]$, for some $\mathbf{x} = \left[x_1, x_2,\ldots,x_{n-2}, 0\right]$, $\mathbf{y} 
= \left[y_1, y_2,\ldots,y_{n-2}, 0\right]^T$, and $a \in \reals$, such that  $A^T\circ \T_{n-1}=0$.  Suppose that $B$ is in the centralizer of $N$. We claim that $B=0$. Observe that
\begin{equation*}\label{NB}
    NB = \left[ \begin{array}{cc}
MA+\mathbf{e}^T\mathbf{x}     & M\mathbf{y} + a\mathbf{e}^T \\
0                               & 0 \end{array} \right]\qquad \mbox{\rm and \qquad }
    BN = \left[ \begin{array}{cc}
AM     & A\mathbf{e}^T \\
\mathbf{x}M & \mathbf{x}\mathbf{e}^T \end{array} \right].
\end{equation*}
Let $L$ be the principal submatrix of $M$ consisting of the first $n-1$ rows and columns.
Since $L$ has exactly one transversal, $L$ is nonsingular. With $NB=BN$, it follows that $\mathbf{x}M = \mathbf{0}$ 
and hence $\mathbf{x}=\mathbf{0}$ since $L$ is nonsingular. 
Since $\mathbf{x} = \mathbf{0}$, $MA = AM$, and since the only matrix $A$ in the centralizer of $M$ with $A^T\circ\T_{n-1}=0$ is $0$, $A = 0$. This implies that $M\mathbf{y} + a\mathbf{e}^T = \mathbf{0}$. 
Ignoring the last row of this matrix equation gives us the equation
$L[y_1,\ldots,y_{n-2}]^T=\mathbf{0}$. 
Since $L$ is invertible,
we have $\mathbf{y}=\mathbf{0}$.
It follows that $a=0$ and hence $B=0$.
Therefore $\W_{n}$ is spectrally arbitrary for all $n\geq 3$ by Theorem~\ref{NC}.
\end{proof}

\begin{remark}{\rm 
While $\T_n$ has a signing that is spectrally
arbitrary~\cite{GS},  there is no signing of the nonzero
entries of $\W_n$ that is spectrally arbitrary. In particular,
$\W_n$ has exactly one nonzero transversal, so a nilpotent realization
requires a zero in one of the $\oast$ positions.
}\end{remark}

\section{Classifying patterns of order 3}\label{order3}

By Lemma~\ref{kcycle}, if $\A$ is an order $3$ irreducible pattern that is inertially arbitrary, then $D(\A)$ must have at least one 
loop and either a proper two-cycle or another loop, in order to have a composite $2$-cycle. Hence, accounting for irreducibility, we have
the following: 

\begin{lemma}\label{min5}
If $\A$ is an irreducible inertially arbitrary zero pattern of order $3$, then
$\A$ must have at least $5$ $\oast$ entries.
\end{lemma}

The next result describes the irreducible spectrally arbitrary patterns of order $3$ as demonstrated in~\cite{ESV}. 

\begin{theorem}\cite[Corollary~3.11]{ESV}\label{SAP3}
Given $\A$ is an 
irreducible zero pattern of order $3$, then $\A$ is spectrally arbitrary if and only if $\A$ is equivalent to 
a superpattern of
$$\left[ \begin{array}{ccc}
\oast & \oast & 0 \\
\oast & 0 & \oast \\
\oast & 0 & 0\end{array} \right], \quad
\left[ \begin{array}{ccc}
\oast & \oast & 0 \\
0 & 0 & \oast \\
\oast & \oast & 0\end{array} \right], \quad 
\left[ \begin{array}{ccc}
\oast & \oast & 0 \\
\oast & 0 & \oast \\
0  & \oast & \oast \end{array} \right], \quad {\rm or} \quad 
\left[ \begin{array}{ccc}
\oast & \oast & 0 \\
\oast & \oast & \oast \\
0  & \oast & 0 \end{array} \right].$$
\end{theorem}
Next we classify the irreducible patterns of order $3$ that are refined inertially arbitrary.  

\begin{theorem}\label{RIAP3}
Given $\A$ is an irreducible order $3$ pattern, then  
$\A$ is refined inertially arbitrary   
if and only if $\A$ is equivalent to a superpattern of either 
$$\left[ \begin{array}{ccc}
\oast & \oast & 0 \\
\oast & 0 & \oast \\
\oast & 0 & 0\end{array} \right], \quad 
\left[ \begin{array}{ccc}
\oast & \oast & 0 \\
0 & 0 & \oast \\
\oast & \oast & 0\end{array} \right],\quad {\rm or} \quad 
\left[ \begin{array}{ccc}
\oast & \oast & 0 \\
\oast & 0 & \oast \\
0 & \oast & 0 \end{array} \right].
$$
\end{theorem}

\begin{proof}
Suppose $\A$ is refined inertially arbitrary.
If $\A$ is irreducible, then $\A$ is a superpattern of either 
$$\HH_1=\left[ \begin{array}{ccc}
0 & \oast & 0 \\
0 & 0 & \oast \\
\oast & 0 & 0\end{array} \right], \quad {\rm or} \quad 
\HH_2=\left[ \begin{array}{ccc}
0 & \oast & 0 \\
\oast & 0 & \oast \\
0 & \oast & 0 \end{array} \right].
$$
By Lemma~\ref{2cycle}, $D(\A)$ must have a proper $2$-cycle and a loop. If
 $\A$ is a superpattern pattern of $\HH_1$, then $\A$ is equivalent to a superpattern of one of the first two patterns
 in Theorem~\ref{SAP3}. Suppose $\A$ is a superpattern of $\HH_2$. Then $\A$ is equivalent to a superpattern
 of $\cP_{3,1}$ or $\cP_{3,2}$. If $\A$ is a superpattern of $\cP_{3,1}$ then $\A$ is 
 refined inertially arbitrary by Theorem~\ref{pathRIAP}.
  If $\A=\cP_{3,2}$  then $\A$ is not refined inertially arbitrary by Theorem~\ref{patheorem}. 
 If $\A$ is a proper superpattern of
 $\cP_{3,2}$, then $\A$ is one of the first two patterns in Theorem~\ref{SAP3},
 or $\A$ is a superpattern of $\cP_{3,1}$. 
\end{proof}

\begin{corollary}
If $\A$ is an   
irreducible refined inertially arbitrary pattern of order $3$ that is not spectrally arbitrary, 
then $\A$ is equivalent to $\cP_{3,1}$.  
\end{corollary}

\begin{theorem}\label{minIAP} If an irreducible zero pattern of order $3$ is inertially
arbitrary, but not refined inertially arbitrary, then it is equivalent to
$$\A_3 = \left[ \begin{array}{ccc}
\oast & \oast & 0 \\
0 & \oast & \oast \\
\oast & 0 & 0 \end{array} \right].$$
\end{theorem}

\begin{proof}
Let $\A$ be inertially arbitrary. By Lemma~\ref{cycleprod}, $\A$ has an $\oast$ entry on the
main diagonal.  Note that $\A\neq\cP_{3,2}$ since any matrix with pattern $\cP_{3,2}$ is singular.
If $\A$ is a superpattern of $\cP_{3,1}$ or a proper superpattern of $\cP_{3,2}$,
 then $\A$ is refined inertially arbitrary as noted in the previous
proof. Thus, if $\A$ is not refined inertially arbitrary, then $\A$ is equivalent to a superpattern of the pattern $\B_1$
defined in the proof of Theorem~\ref{RIAP3}. 
 If $D(\A)$ has a proper $2$-cycle, then $\A$ will be equivalent to a superpattern of a
spectrally arbitrary pattern, namely one of the first two patterns in Theorem~\ref{SAP3}. Thus $D(\A)$ does not
have a proper $2$-cycle and hence,  by Lemma~\ref{cycleprod}, $\A$ must have two $\oast$ entries on the main diagonal.  
Note that $\A$ can not have three $\oast$ entries on the diagonal, otherwise it would be 
refined inertially arbitrary as noted in Lemma~\ref{iapdiagonal}.
Thus $D(\A)$ consists of a proper $3$-cycle with two loops. 
\end{proof}

\begin{corollary}
Given $\A$ is an irreducible order $3$ pattern, $\A$ 
is inertially arbitrary if and only if $\A$ is equivalent to a superpattern of 
$$ 
\left[ \begin{array}{ccc}
\oast & \oast & 0 \\
\oast & 0 & \oast \\
\oast & 0 & 0 \end{array} \right],
\left[ \begin{array}{ccc}
\oast & \oast & 0 \\
0 & 0 & \oast \\
\oast & \oast & 0 \end{array} \right], 
\left[ \begin{array}{ccc}
\oast & \oast & 0 \\
\oast & 0 & \oast \\
0 & \oast & 0 \end{array} \right],
\mbox{\rm{\quad or \quad}}
\left[ \begin{array}{ccc}
\oast & \oast & 0 \\
0 & \oast & \oast \\
\oast & 0 & 0 \end{array} \right].$$
\end{corollary}

\begin{proof} 
The result follows from Theorem~\ref{RIAP3} and Theorem~\ref{minIAP}.
\end{proof}

\section{Order 4 refined inertially arbitrary patterns}\label{order4}
\tikzstyle{place}=[circle,draw=black!100,fill=black!100,thick,inner sep=0pt,minimum size=1mm]
\tikzstyle{left}=[>=latex,<-,semithick]
\tikzstyle{right}=[>=latex,->,semithick]
\tikzstyle{nleft}=[>=latex,-,semithick]
\tikzstyle{nright}=[>=latex,-,semithick]
\tikzstyle{double}=[>=latex,<->,semithick]
\tikzstyle{right2}=[-,semithick]
In this section, we determine all refined inertially arbitrary order $4$ zero patterns with the least possible number of $\oast$ entries. In general, it is an open question as to the minimum number of $\oast$ entries
in an inertially or a refined inertially arbitrary pattern. It is known~\cite{CF} that an irreducible spectrally arbitrary pattern of order $n$ requires at least $2n-1$ nonzero entries.
In Lemma~\ref{not6},  we determine that an order $4$ refined inertially arbitrary pattern requires at least
seven $\oast$ entries. Then, in Theorem~\ref{seven}, we determine which of the irreducible patterns 
with seven $\oast$ entries are refined inertially arbitrary.

\begin{lemma}\label{not6}
If $\A$ is an irreducible refined inertially arbitrary 
pattern of order $4$, then $\A$ must have at least seven $\oast$ entries.
\end{lemma}

\begin{proof}
Suppose $\A$ is an irreducible refined inertially arbitrary pattern of order $4$. 
By Lemma~\ref{kcycle} and Lemma~\ref{2cycle}, $D=D(\A)$ must have a 
loop and a proper $2$-cycle. If $D$ had only $2$-cycles and no
proper cycles of higher order, then the underlying graph of $D$ is a tree (having $3$ edges). Thus, since $\A$ is irreducible, 
 $D$ would have at least $3\times 2 +1=7$ edges, including at least one loop. Therefore, assume that $D$ has a proper $k$-cycle for some $k >2$. 

First suppose that $D$ has a proper $4$-cycle. Since $D$ also has a proper $2$-cycle and a loop, 
$\A$ has at least six nonzero entries. Suppose $\A$ has only six nonzero entries. 
If the loop is incident to a proper $2$-cycle, then $D$ does not have a composite $3$-cycle
and so would fail to be inertially arbitrary by Lemma~\ref{kcycle}. Thus the loop is not
incident to the $2$-cycle in $D$. The resulting pattern does not allow refined
inertia $(0,2,0,2)$. In particular, if $A\in Q(\A)$ has inertia $(0,2,0,2)$, then
the characteristic polynomial of $A$ is of the either of the form 
$$(x^2+2\alpha x+\alpha^2+\beta^2)(x^2+\omega^2)
= x^4+2\alpha x^3+(\alpha^2+\beta^2+\omega^2)x^2+2\alpha \omega^2 x+ (\alpha^2+\beta^2)\omega^2$$
for some positive real numbers $\alpha$ and $\omega$, and non-negative real number $\beta$, or of the form
$$(x+\gamma)(x+\kappa)(x^2+\omega^2)=x^4+(\gamma+\kappa)x^3+(\gamma\kappa+\omega^2)x^2+(\gamma+\kappa)\omega^2x
+\gamma\kappa\omega^2,
$$
for some positive real numbers $\gamma, \kappa$ and $\omega$. 
However, the only composite $3$-cycle in $D$ is obtained by combining the loop with
the $2$-cycle, and hence by Lemma~\ref{cycleprod}, the coefficient of $x$ must by
the product of the coefficients of $x^3$ and $x^2$. This would imply that
$2\alpha(\alpha^2+\beta^2)=0$ in the first case, and 
$(\gamma+\kappa)(\gamma\kappa)=0$ in the latter case. Both of these would
be contradictions since $\alpha, \gamma$ and $\kappa$ are positive.

Suppose that $D$ does not have a proper $4$-cycle, but instead has a proper $3$-cycle. Assuming $\A$ has
less than seven $\oast$ entries, then $D$ can not have two proper $3$-cycles as well as a loop and
a proper $2$-cycle. Thus $D$ must have exactly one proper $3$-cycle and since $\A$ is irreducible, $D$ must have
the digraph in Figure~\ref{D}, with the loop placed so that $D$ has a composite
$4$-cycle (as required by Lemma~\ref{kcycle}). However, in this case, if the coefficient of $x^3$
is zero in the characteristic polynomial of $A\in Q(\A)$, then $\det(A)=0$. This
would imply that $\A$ does not allow refined inertia $(0,0,0,4)$.  
\begin{figure}[ht]
${\tikzpicture \phantom{\node (5) at (-0.25,0.28)[place]{};}
\node (1) at (-0.5,0.5)[place] {};
\node (2) at (0.5,0.5)[place] {};
\node (3) at (0.5,-0.5)[place] {};
\node (4) at (-0.5,-0.5)[place] {};
\draw [nright] (3) to [bend right=10] (4);
\draw [nright] (4) to [bend right=10] (3);
\draw [right] (1) to (2);
\draw [right] (2) to (3);
\draw [right] (3) to (1);
\draw [-] (-0.5,-0.55) arc (360:0:3pt);
\endtikzpicture}$\caption{}\label{D}
\end{figure}

 Therefore, an irreducible refined inertially arbitrary pattern must have at least
 seven $\oast$ entries.
\end{proof}

\begin{figure}[hb]
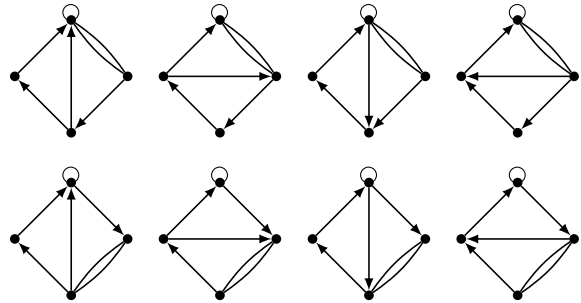

{\tikzpicture[baseline=-5] 
						\node (1) at (0, 0.75)[place] {};
						\node (2) at (0.75, 0)[place] {};
						\node (3) at (0,-0.75)[place] {};
						\node (4) at (-0.75,0)[place] {};
						\draw [nright] (2) to [bend right=10] (1);
						\draw [nright] (1) to [bend right=10] (2);
						\draw [right] (3) to (1);
						\draw [right] (4) to (1);
						\draw [right] (2) to (3);
						\draw [right] (3) to (4);
						\draw [-] (0.1, 0.85) arc (360:0:3pt);
						\endtikzpicture \quad
\tikzpicture[baseline=-5] 
						\node (1) at (0, 0.75)[place] {};
						\node (2) at (0.75, 0)[place] {};
						\node (3) at (0,-0.75)[place] {};
						\node (4) at (-0.75,0)[place] {};
						\draw [nright] (2) to [bend right=10] (1);
						\draw [nright] (1) to [bend right=10] (2);
						\draw [right] (4) to (2);
						\draw [right] (4) to (1);
						\draw [right] (2) to (3);
						\draw [right] (3) to (4);
						\draw [-] (0.1, 0.85) arc (360:0:3pt);
						\endtikzpicture \quad 
	\tikzpicture[baseline=-5] 
						\node (1) at (0, 0.75)[place] {};
						\node (2) at (0.75, 0)[place] {};
						\node (3) at (0,-0.75)[place] {};
						\node (4) at (-0.75,0)[place] {};
						\draw [nright] (2) to [bend right=10] (1);
						\draw [nright] (1) to [bend right=10] (2);
						\draw [right] (1) to (3);
						\draw [right] (4) to (1);
						\draw [right] (2) to (3);
						\draw [right] (3) to (4);
						\draw [-] (0.1, 0.85) arc (360:0:3pt);
						\endtikzpicture\quad		
												\tikzpicture[baseline=-5]
						\node (1) at (0, 0.75)[place] {};
						\node (2) at (0.75, 0)[place] {};
						\node (3) at (0,-0.75)[place] {};
						\node (4) at (-0.75,0)[place] {};
						\draw [nright] (2) to [bend right=10] (1);
						\draw [nright] (1) to [bend right=10] (2);
						\draw [right] (2) to (4);
						\draw [right] (4) to (1);
						\draw [right] (2) to (3);
						\draw [right] (3) to (4);
						\draw [-] (0.1, 0.85) arc (360:0:3pt);
						\endtikzpicture}\vspace{1em}
						
	{\tikzpicture[baseline=-5]
						\node (1) at (0, 0.75)[place] {};
						\node (2) at (0.75, 0)[place] {};
						\node (3) at (0,-0.75)[place] {};
						\node (4) at (-0.75,0)[place] {};
						\draw [nright] (2) to [bend right=10] (3);
						\draw [nright] (3) to [bend right=10] (2);
						\draw [right] (1) to (2);
						\draw [right] (3) to (1);
						\draw [right] (3) to (4);
						\draw [right] (4) to (1);
						\draw [-] (0.1, 0.85) arc (360:0:3pt);
						\endtikzpicture\quad
						\tikzpicture[baseline=-5]
						\node (1) at (0, 0.75)[place] {};
						\node (2) at (0.75, 0)[place] {};
						\node (3) at (0,-0.75)[place] {};
						\node (4) at (-0.75,0)[place] {};
						\draw [nright] (2) to [bend right=10] (3);
						\draw [nright] (3) to [bend right=10] (2);
						\draw [right] (1) to (2);
						\draw [right] (3) to (4);
						\draw [right] (4) to (1);
						\draw [right] (4) to (2);
						\draw [-] (0.1, 0.85) arc (360:0:3pt);
						\endtikzpicture\quad
						\tikzpicture[baseline=-5]
						\node (1) at (0, 0.75)[place] {};
						\node (2) at (0.75, 0)[place] {};
						\node (3) at (0,-0.75)[place] {};
						\node (4) at (-0.75,0)[place] {};
						\draw [nright] (2) to [bend right=10] (3);
						\draw [nright] (3) to [bend right=10] (2);
						\draw [right] (1) to (2);
						\draw [right] (3) to (4);
						\draw [right] (4) to (1);
						\draw [right] (1) to (3);
						\draw [-] (0.1, 0.85) arc (360:0:3pt);
						\endtikzpicture\quad
						\tikzpicture[baseline=-5]
						\node (1) at (0, 0.75)[place] {};
						\node (2) at (0.75, 0)[place] {};
						\node (3) at (0,-0.75)[place] {};
						\node (4) at (-0.75,0)[place] {};
						\draw [nright] (2) to [bend right=10] (3);
						\draw [nright] (3) to [bend right=10] (2);
						\draw [right] (1) to (2);
						\draw [right] (3) to (4);
						\draw [right] (4) to (1);
						\draw [right] (2) to (4);
						\draw [-] (0.1, 0.85) arc (360:0:3pt);
						\endtikzpicture}				
\caption{Digraphs of spectrally arbitrary patterns in $\ccC_4$}\label{C4}							
\end{figure}

\bigskip

Let $\ccC_n$ be the set of zero patterns with a digraph $D$ satisfying:~\begin{enumerate}
\item $D$ has exactly one $k$-cycle for each $k$, $1\leq k\leq n$,
\item $D$ has a directed Hamilton path, and
\item each proper cycle of $D$ contains exactly one arc not on the Hamilton path.
\end{enumerate}
Figure~\ref{C4} lists the digraphs corresponding to
the patterns in $\ccC_4$. Each of these patterns 
are spectrally arbitrary: 
\begin{lemma}\cite[Theorem 2.2]{ESV}\label{C4L}
If $\A$ is a pattern in $\ccC_n$, then
$\A$ is spectrally arbitrary. 
\end{lemma} 

\begin{remark}\label{YSAP}
{\rm The first digraph in Figure~\ref{Y} corresponds to the 
spectrally arbitrary pattern $\Y_4(3,1)$ from \cite{ESV}.
The second digraph in Figure~\ref{Y} has
the same cycle adjacency structure as $\Y_4(3,1)$ and
hence by Lemma~\ref{cycleprod}, it is also spectrally
arbitrary.} 
\end{remark}
\begin{figure}[h]
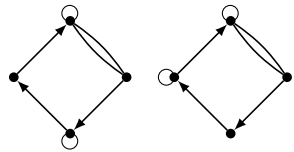

\tikzpicture[baseline=-5]
				\node (1) at (0, 0.75)[place] {};
				\node (2) at (0.75, 0)[place] {};
				\node (3) at (0,-0.75)[place] {};
				\node (4) at (-0.75,0)[place] {};
				\draw [nright] (1) to [bend right=10] (2);
				\draw [nright] (2) to [bend right=10] (1);
				\draw [right] (2) to (3);
				\draw [right] (3) to (4);
				\draw [right] (4) to (1);
				\draw [-] (0.1, 0.85) arc (360:0:3pt);
				\draw [-] (0.1, -0.85) arc (360:0:3pt);
				\endtikzpicture\quad
\tikzpicture[baseline=-5]
				\node (1) at (0, 0.75)[place] {};
				\node (2) at (0.75, 0)[place] {};
				\node (3) at (0,-0.75)[place] {};
				\node (4) at (-0.75,0)[place] {};
				\draw [nright] (1) to [bend right=10] (2);
				\draw [nright] (2) to [bend right=10] (1);
				\draw [right] (2) to (3);
				\draw [right] (3) to (4);
				\draw [right] (4) to (1);
				\draw [-] (0.1, 0.85) arc (360:0:3pt);
				\draw [-] (-0.75, 0) arc (360:0:3pt);
				\endtikzpicture
				\caption{Digraphs of $\Y_4(3,1)$ and a related spectrally arbitrary pattern.}\label{Y}
\end{figure}

\bigskip
The next theorem references the tables in the Appendix.
\begin{theorem}\label{seven}
If $\A$ is an irreducible order $4$ pattern with $7$ nonzero entries, then
\begin{enumerate}
\item $\A$ is spectrally arbitrary if and only if $\A$ is equivalent a pattern with a digraph in Figure~$\ref{C4}$ 
or Figure~$\ref{Y}$. 
\item\label{part2} $\A$ is refined inertially arbitrary, but not spectrally arbitrary, 
if and only if $\A$ is equivalent to $\cP_{4,1}$, $\cP_{4,2}$, or a pattern in Table~$\ref{RIAPT}$.
\item\label{part3} $\A$ is inertially arbitrary, but not refined inertially arbitrary,
if and only if $\A$ is equivalent to $\A_4$, a pattern in 
Table~$\ref{IAPT}$ or a pattern in Table~\ref{IAP4LastCase}. 
\end{enumerate}
\end{theorem}

\begin{proof}  The tables in the Appendix, along with Lemma~\ref{C4L}, Remark~\ref{YSAP}, and Theorem~\ref{patheorem}, 
provide justification for part of the characterization. That the patterns in (\ref{part2}) are not spectrally arbitrary
and the patterns in (\ref{part3}) are not refined inertially arbitrary will be noted in the casework 
as we justify the converse of the characterization. 

Let $\A$ be an irreducible inertially arbitrary pattern of order $4$ with $7$ $\oast$ entries.
Then the digraph of $A$ has at least one loop and at least one $2$-cycle by Lemma~\ref{kcycle}. 
 
	\textbf{Case 1:} Suppose $D(\A)$ has exactly one loop. Then $D(\A)$ has a proper $2$-cycle.
	\begin{enumerate}[(A)]
		\item Suppose $D(\A)$ has exactly one proper $2$-cycle. 
		\begin{enumerate}[(I)]
			\item Suppose the loop is incident to the proper $2$-cycle. By Lemma~\ref{cycleprod}, $D(\A)$ must
			have a composite $4$-cycle in order to realize inertia $(4,0,0)$.
			\begin{enumerate}[(a)]
				\item Suppose $D(\A)$ has a proper $4$-cycle. 
				\begin{enumerate}[(i)]
					\item Suppose the proper $2$-cycle shares an arc with a proper $4$-cycle. Then $\A$ is in $\ccC_4$ and
				$\A$ is spectrally arbitrary by Lemma~\ref{C4L}. These are the first four patterns from Figure~\ref{C4}. 
					\item Suppose the proper $2$-cycle does not share an arc with a proper $4$-cycle. Then 
					$\A$ is equivalent to $\B_1$ in Table~\ref{IAPT}.  
					$\B_1$ is not refined inertially arbitrary since it does not allow the
					 refined inertia $(0,0,0,4)$ since there exists a superpattern (see \cite[Example 2.16]{CF}) 
					 of this pattern that does not  allow refined inertia $(0,0,0,4)$. 
				\end{enumerate}
				\item Suppose $D(\A)$ has no proper $4$-cycle. Then $\A$ is equivalent to $\B_2$
				in Table~\ref{IAPT}.   
				$\B_2$ can not realize the refined inertia $(0,0,0,4)$ 
				since for any $A\in Q(\B_2)$, if the trace of $A$ is zero, then $\det(A)=0$.		 
			\end{enumerate}
			\item Suppose the loop of $D(\A)$ is not incident to the $2$-cycle. 
			\begin{enumerate}[(a)]
				\item Suppose $D(\A)$ has  a proper $4$-cycle. 
				\begin{enumerate}[(i)]
					\item Suppose the proper $2$-cycle shares an arc with a $4$-cycle. Then $\A$ is in $\ccC_4$ and
				$\A$ is spectrally arbitrary by Lemma~\ref{C4L}. These correspond to the second four digraphs in Figure~\ref{C4}. 
					
					\item  Suppose the proper $2$-cycle does not share an arc with a $4$-cycle.
					Then $\A$ is equivalent to $\B_3$ in Table~\ref{IAPT}.  
					$\B_3$ 	does not allow refined inertia $(0,0,0,4)$ and so is not refined
					inertially arbitrary.  In particular, one can check that if
					$A\in Q(\B_3)$ has characteristic polynomial 
					$p(x)=x^4+E_2x^2+E_4$, with $E_2$ and $E_4$ nonzero,
					then $E_2$ and $E_4$ must have opposite signs.
				\end{enumerate}
				\item Suppose $D(\A)$ has no proper $4$-cycle. Then $\A$
				is equivalent to either  $\B_4$ or $\B_5$ in Table~\ref{IAPT}. 
				There is no matrix $A\in Q(\B_4)\cup Q(\B_5)$ with refined inertia $(0,0,0,4)$ 
				since if the trace of $A$ is zero, then $\det(A)=0$. Thus, these
				two patterns are not refined inertially arbitrary.
				\end{enumerate}
		\end{enumerate}
		\item Suppose $D(\A)$ has exactly two proper $2$-cycles. 
		\begin{enumerate}[(I)]
			\item Suppose both proper $2$-cycles are incident to the loop. Considering $\A$ is 
			irreducible and has seven nonzero entries, $\A$ would need
			to be equivalent to one of the first two patterns in Figure~\ref{nIAPD}. But the digraph of 
			first pattern  has no composite $3$-cycle and the digraph of the second pattern has no composite $4$-cycle. 
			Thus, by Lemma~\ref{kcycle}, neither of these patterns are inertially arbitrary.  
					\begin{figure}
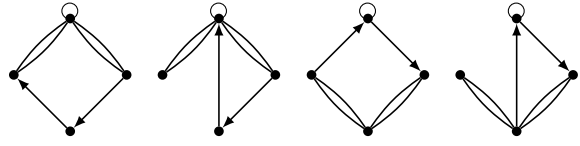

					{\tikzpicture[baseline=-5]
					\node (1) at (0, 0.75)[place] {};
					\node (2) at (0.75, 0)[place] {};
					\node (3) at (0,-0.75)[place] {};
					\node (4) at (-0.75,0)[place] {};
					\draw [nright] (1) to [bend right=10] (2);
					\draw [nright] (2) to [bend right=10] (1);
					\draw [nright] (1) to [bend right=10] (4);
					\draw [nright] (4) to [bend right=10] (1);
					\draw [right] (2) to (3);
					\draw [right] (3) to (4);
					\draw [-] (0.1, 0.85) arc (360:0:3pt);
					\endtikzpicture\quad
										\tikzpicture[baseline=-5]
					\node (1) at (0, 0.75)[place] {};
					\node (2) at (0.75, 0)[place] {};
					\node (3) at (0,-0.75)[place] {};
					\node (4) at (-0.75,0)[place] {};
					\draw [nright] (1) to [bend right=10] (2);
					\draw [nright] (2) to [bend right=10] (1);
					\draw [nright] (1) to [bend right=10] (4);
					\draw [nright] (4) to [bend right=10] (1);
					\draw [right] (2) to (3);
					\draw [right] (3) to (1);
					\draw [-] (0.1, 0.85) arc (360:0:3pt);
					\endtikzpicture\quad
\tikzpicture[baseline=-5]
					\node (1) at (0, 0.75)[place] {};
					\node (2) at (0.75, 0)[place] {};
					\node (3) at (0,-0.75)[place] {};
					\node (4) at (-0.75,0)[place] {};
					\draw [nright] (3) to [bend right=10] (2);
					\draw [nright] (2) to [bend right=10] (3);
					\draw [nright] (3) to [bend right=10] (4);
					\draw [nright] (4) to [bend right=10] (3);
					\draw [right] (1) to (2);
					\draw [right] (4) to (1);
					\draw [-] (0.1, 0.85) arc (360:0:3pt);
					\endtikzpicture\quad					
\tikzpicture[baseline=-5]
					\node (1) at (0, 0.75)[place] {};
					\node (2) at (0.75, 0)[place] {};
					\node (3) at (0,-0.75)[place] {};
					\node (4) at (-0.75,0)[place] {};
					\draw [nright] (3) to [bend right=10] (2);
					\draw [nright] (2) to [bend right=10] (3);
					\draw [nright] (3) to [bend right=10] (4);
					\draw [nright] (4) to [bend right=10] (3);
					\draw [right] (1) to (2);
					\draw [right] (3) to (1);
					\draw [-] (0.1, 0.85) arc (360:0:3pt);
					\endtikzpicture				
					}
\caption{Digraphs of patterns that are not inertially arbitrary}\label{nIAPD}					
\end{figure}
			\item Suppose $D(\A)$ has one proper $2$-cycle incident to the loop and the proper $2$-cycles are incident to each other. 
			Since $D(\A)$ must have a composite $4$-cycle, $\A$ must be equivalent to 
			$\R_1$ in Table~\ref{RIAPT} or $\B_6$ Table~\ref{IAPT}. 
					$\R_1$ is not spectrally arbitrary since if $A\in Q(\R_1)$ 
					 with trace zero, then the sum of the composite $3$-cycles is zero. 
					 $\B_6$ is not refined inertially arbitrary, since if $A\in Q(\B_6)$ 
					 with trace zero, then $\det(A)=0$.

			\item Suppose one proper $2$-cycle of $D(\A)$ is incident to the loop but the proper $2$-cycles 
			are not incident to each other. Then $\A$ is equivalent to one of the
			four patterns $\R_2, \R_3, \R_4, \R_5$ in Table~\ref{RIAPT}. The patterns
			$\R_2$ and $\R_5$ are not spectrally arbitrary since if $A\in Q(\R_2)\cup Q(\R_5)$, 
			with the sum of the composite $2$-cycles zero, then $\det(A)\neq 0$.
			The pattern $\R_3$ is not spectrally arbitrary since if $A\in Q(\R_3)$
			has trace zero, then the sum of the composite $3$-cycles is zero. The 
			pattern $\R_4$ is not spectrally arbitrary since if $A\in Q(\R_4)$
			with trace zero, $\det(A)\neq 0$, and the sum of the composite
			$2$-cycles is zero, then $\det(A)<0$.
			
			\item Suppose both proper $2$-cycles are not incident to the loop. Then, since
			$\A$ is irreducible, $D(\A)$ must be the third or fourth digraph 
			in Figure~\ref{nIAPD}. The fourth digraph has no composite
			 $4$-cycle, so its corresponding pattern is not inertially arbitrary
			 by Lemma~\ref{kcycle}. One can check that
			 the third digraph corresponds to a pattern that  does not allow inertia $(2, 0, 2)$.

		\end{enumerate}
		\item Suppose $D(\A)$ has three proper $2$-cycles. In this case, the underlying graph of 
		$D(\A)$ must be a tree. If the tree is a star, then the digraph will
		have no composite $4$-cycle, thus the tree must be a path. In this case,
		$\A$ is equivalent to either the refined inertially arbitrary 
		pattern $\cP_{4,1}$ or $\cP_{4,2}$.
		
	\end{enumerate}
	\textbf{Case 2:} Suppose $D(\A)$ has exactly two loops. Note that there can be at most one proper $2$-cycle in this case, 
	 since there are only seven arcs and the digraph must be strongly connected. 
	\begin{enumerate}[(A)]
		\item Suppose $D(\A)$ has a proper $2$-cycle, and it is incident to both loops.
		Since $D(\A)$ must have a composite $3$-cycle, $\A$ will
		be equivalent to $\B_7$ in Table~\ref{IAPT}. In this
		case, if $A\in Q(\B_7)$ and the sum of the composite $3$-cycles of $A$ is zero, 
		then $\det(A)=0$. Thus $\A$ does not allow refined inertia $(0,0,0,4)$.
				
		\item Suppose $D(\A)$ has a proper $2$-cycle, and it is incident to exactly one loop. Then $\A$ is equivalent to either
		one of the spectrally arbitrary patterns in Figure~\ref{Y} or the 
		pattern $\R_6$ in Table~\ref{RIAPT}. If  $A\in Q(\R_6)$, and if the characteristic polynomial
		of $A$ is $x^4+c$, then $c\geq 0$. Thus $\R_6$ is not spectrally arbitrary.  
		
		\item Suppose $D(\A)$ has a proper $2$-cycle, and it is not incident to either loop. Then $\A$ is
		equivalent to either $\R_7$ in Table~\ref{RIAPT} or $\B_8$
		in Table~\ref{IAPT}. If $A\in Q(\R_7)$ with
		trace zero, then the sum of the composite $3$-cycles is zero too and
		so $\R_7$ is not spectrally arbitrary.  
		If $A\in\Q(\R_8)$
		with trace zero and the sum of the composite $3$-cycles is zero, then
		the sums of the composite $2$-cycles and $4$-cycles cannot both
		be positive. Thus, $\R_8$ does not allow refined inertia $(0,0,0,4)$.
		\item Suppose $D(\A)$ has no proper $2$-cycle. Then $\A$ is not refined inertially
		arbitrary by Lemma~\ref{2cycle}.
		\begin{enumerate}
		\item[(I)] Suppose $D(\A)$ has a proper $4$-cycle. Then $D(\A)$ will have
		  exactly one proper $3$-cycle. If only one loop is incident to
		  the proper $3$-cycle, then $D(\A)$ is equivalent to one of the
		  patterns $\J_1$ and $\J_4$ in Table~\ref{IAP4LastCase}. If both loops are incident to the
		  $3$-cycle, then $D(\A)$ is equivalent to $\J_2$ or $\J_3$.
		\item[(II)] Suppose $D(\A)$ has no proper $4$-cycle. If both loops are incident
		to both $3$-cycles, then $D(\A)$ has no composite $4$-cycle in which
		case $\A$ is not inertially arbitrary by Lemma~\ref{kcycle}.  
		 If both loops are on one of the proper $3$-cycles, then $D(\A)$
		 is equivalent to $\J_5$ in Table~\ref{IAP4LastCase}. 
		 If each $3$-cycle is incident to exactly one loop, then $D(\A)$
		 is equivalent to $\J_6$ in Table~\ref{IAP4LastCase}.
		\end{enumerate}
			\end{enumerate}

		\textbf{Case 3:} Suppose $D(\A$) has exactly three loops. Then $\A$ is equivalent to $\A_4$
		and $\A$ is inertially arbitrary but not refined inertially arbitrary by Theorem~\ref{IAPN}.	
			\hfill\end{proof}
\section{Concluding remarks}

In Section~\ref{order3}, we characterized the refined inertially arbitrary patterns of order three. In
Section~\ref{order4}, we determined which of the sparse irreducible patterns of order $4$ 
are spectrally arbitrary, which are refined inertially arbitrary and which are simply inertially arbitrary.
Characterizing all irreducible refined inertially arbitrary patterns of order $4$ is still an open problem. 
Such a project would be of interest especially if it involved a development of new techniques to determine
a refined inertially arbitrary pattern. 
It would also be of interest to explore characteristics of reducible refined inertially arbitrary patterns. 
These could be important for building irreducible patterns, since, for zero patterns, the eigenvalue 
properties that a pattern allows are preserved for superpatterns.

\section{Appendix}

This appendix provides the data for Theorem~\ref{seven}.
Table~\ref{RIAPT} lists digraphs of patterns which are refined inertially arbitrary; for each pattern, we list specific matrices that realize the various refined inertias. Since a pattern is preserved under matrix negation,
we only list matrices having inertias with $n_+\geq n_-$. Refined inertially arbitrary patterns 
$\cP_{4,1}$ and $\cP_{4,2}$ are not represented in Table~\ref{RIAPT}, nor are spectrally arbitrary patterns. 
 Table~\ref{IAPT} and Table~\ref{IAP4LastCase} list digraphs of patterns which are inertially arbitrary but not refined inertially arbitrary;  for each pattern we list specific matrices that realize the various inertias. 
 This data is used in the proof of Theorem~\ref{seven}. 

\scriptsize
\setlength{\tabcolsep}{.02in}
\begin{longtable}{rrlrrrrrrrrrrrrrrrrr}
\caption{Refined inertially arbitrary patterns.}\label{RIAPT}\\

\hline
\hline
\multicolumn{2}{c}{Pattern}  & \multicolumn{17}{c}{Matrix} & Inertia \\
\hline
\hline

\multirow{5}{*}{$\R_1$}
&   \multirow{5}{*}{\tikzpicture \phantom{\node (5) at (-0.25,0.28)[place]{};}
\node (1) at (-0.5,0.5)[place] {};
\node (2) at (0.5,0.5)[place] {};
\node (3) at (0.5,-0.5)[place] {};
\node (4) at (-0.5,-0.5)[place] {};
\draw [nright] (1) to [bend right=10] (2);
\draw [nright] (2) to [bend right=10] (1);
\draw [nright] (3) to [bend right=10] (2);
\draw [nright] (2) to [bend right=10] (3);
\draw [right] (3) to (4);
\draw [right] (4) to (1);
\draw [-] (-0.5,0.55) arc (360:0:3pt);
\endtikzpicture}
& $[$& 1 & 1& 0& 0;& $-1$& 0& 1& 0;& 0& $-2$& 0& 1;& $-1$& 0& 0& $\left. 0\right]$ & $\left(4, 0, 0, 0\right)$\\
& & $[$& -1 & 1& 0& 0;& 1& 0& 1& 0;& 0& $-2$& 0& 1;& 1& 0& 0& $\left. 0\right]$ & $\left(3, 1, 0, 0\right)$\\
& & $[$& 1 & 1& 0& 0;& $-1$& 0& 1& 0;& 0& $-1$& 0& 1;& 0& 0& 0& $\left. 0\right]$ & $\left(3, 0, 1, 0\right)$\\
& & $[$& -1 & 1& 0& 0;& $-1$& 0& 1& 0;& 0& $-1$& 0& 1;& $-2$& 0& 0& $\left. 0\right]$ & $\left(2, 2, 0, 0\right)$\\
& & $[$& -1 & 1& 0& 0;& 1& 0& 1& 0;& 0& $-1$& 0& 1;& 0& 0& 0& $\left. 0\right]$ & $\left(2, 1, 1, 0\right)$\\
& & $[$& 1 & 1& 0& 0;& $-1$& 0& 1& 0;& 0& 0& 0& 1;& 0& 0& 0& $\left. 0\right]$ & $\left(2, 0, 2, 0\right)$\\
& & $[$& 1 & 1& 0& 0;& $-1$& 0& 1& 0;& 0& $-1$& 0& 1;& $-1$& 0& 0& $\left. 0\right]$ & $\left(2, 0, 0, 2\right)$\\
& & $[$& -1 & 1& 0& 0;& 1& 0& 1& 0;& 0& 0& 0& 1;& 0& 0& 0& $\left. 0\right]$ & $\left(1, 1, 2, 0\right)$\\
& & $[$& -1 & 1& 0& 0;& 1& 0& 1& 0;& 0& $-1$& 0& 1;& 1& 0& 0& $\left. 0\right]$ & $\left(1, 1, 0, 2\right)$\\
& & $[$& 1 & 1& 0& 0;& 0& 0& 1& 0;& 0& 0& 0& 1;& 0& 0& 0& $\left. 0\right]$ & $\left(1, 0, 3, 0\right)$\\
& & $[$& 1 & 1& 0& 0;& 0& 0& 1& 0;& 0& $-1$& 0& 1;& 0& 0& 0& $\left. 0\right]$ & $\left(1, 0, 1, 2\right)$\\
& & $[$& 0 & 1& 0& 0;& $-1$& 0& 1& 0;& 0& 1& 0& 1;& 0& 0& 0& $\left. 0\right]$ & $\left(0, 0, 4, 0\right)$\\
& & $[$& 0 & 1& 0& 0;& $-1$& 0& 1& 0;& 0& $-1$& 0& 1;& 0& 0& 0& $\left. 0\right]$ & $\left(0, 0, 2, 2\right)$\\
& & $[$& 0 & 1& 0& 0;& $-1$& 0& 1& 0;& 0& $-1$& 0& 1;& $-1$& 0& 0& $\left. 0\right]$ & $\left(0, 0, 0, 4\right)$\\
\hline

\multirow{5}{*}{$\R_2$}
&   \multirow{5}{*}{\tikzpicture \phantom{\node (5) at (-0.25,0.28)[place]{};}
\node (1) at (-0.5,0.5)[place] {};
\node (2) at (0.5,0.5)[place] {};
\node (3) at (0.5,-0.5)[place] {};
\node (4) at (-0.5,-0.5)[place] {};
\draw [nright] (1) to [bend right=10] (2);
\draw [nright] (2) to [bend right=10] (1);
\draw [nright] (3) to [bend right=10] (4);
\draw [nright] (4) to [bend right=10] (3);
\draw [right] (2) to (3);
\draw [right] (3) to (1);
\draw [-] (-0.5,0.55) arc (360:0:3pt);
\endtikzpicture}
& $[$& 2 & 1& 0& 0;& $-1$& 0& $-1$& 0;& 1& 0& 0& $-2;$& 0& 0& 1& $\left. 0\right]$ & $\left(4, 0, 0, 0\right)$\\
& & $[$& -1 & 1& 0& 0;& 1& 0& $-1$& 0;& 1& 0& 0& $-1;$& 0& 0& 1& $\left. 0\right]$ & $\left(3, 1, 0, 0\right)$\\
& & $[$& 1 & 1& 0& 0;& 0& 0& $-1$& 0;& 1& 0& 0& $-2;$& 0& 0& 1& $\left. 0\right]$ & $\left(3, 0, 1, 0\right)$\\
& & $[$& -1 & 1& 0& 0;& $-1$& 0& $-1$& 0;& 1& 0& 0& $-1;$& 0& 0& 1& $\left. 0\right]$ & $\left(2, 2, 0, 0\right)$\\
& & $[$& -1 & 1& 0& 0;& $-1$& 0& $-2$& 0;& 1& 0& 0& 0;& 0& 0& 1& $\left. 0\right]$ & $\left(2, 1, 1, 0\right)$\\
& & $[$& 1 & 1& 0& 0;& $-1$& 0& 0& 0;& 1& 0& 0& 0;& 0& 0& 1& $\left. 0\right]$ & $\left(2, 0, 2, 0\right)$\\
& & $[$& 1 & 1& 0& 0;& $-1$& 0& $-1$& 0;& 1& 0& 0& $-2;$& 0& 0& 1& $\left. 0\right]$ & $\left(2, 0, 0, 2\right)$\\
& & $[$& -1 & 1& 0& 0;& 0& 0& $-1$& 0;& 1& 0& 0& 1;& 0& 0& 1& $\left. 0\right]$ & $\left(1, 1, 2, 0\right)$\\
& & $[$& -1 & 1& 0& 0;& $-1$& 0& $-2$& 0;& 1& 0& 0& 1;& 0& 0& 1& $\left. 0\right]$ & $\left(1, 1, 0, 2\right)$\\
& & $[$& 1 & 1& 0& 0;& 0& 0& 0& 0;& 1& 0& 0& 0;& 0& 0& 1& $\left. 0\right]$ & $\left(1, 0, 3, 0\right)$\\
& & $[$& 1 & 1& 0& 0;& $-1$& 0& 1& 0;& 1& 0& 0& 0;& 0& 0& 1& $\left. 0\right]$ & $\left(1, 0, 1, 2\right)$\\
& & $[$& 0 & 1& 0& 0;& 0& 0& 0& 0;& 1& 0& 0& 0;& 0& 0& 1& $\left. 0\right]$ & $\left(0, 0, 4, 0\right)$\\
& & $[$& 0 & 1& 0& 0;& $-1$& 0& 0& 0;& 1& 0& 0& 0;& 0& 0& 1& $\left. 0\right]$ & $\left(0, 0, 2, 2\right)$\\
& & $[$& 0 & 1& 0& 0;& $-1$& 0& 0& 0;& 1& 0& 0& $-1;$& 0& 0& 1& $\left. 0\right]$ & $\left(0, 0, 0, 4\right)$\\
\hline

\multirow{5}{*}{$\R_3$}
&   \multirow{5}{*}{\tikzpicture \phantom{\node (5) at (-0.25,0.28)[place]{};}
\node (1) at (-0.5,0.5)[place] {};
\node (2) at (0.5,0.5)[place] {};
\node (3) at (0.5,-0.5)[place] {};
\node (4) at (-0.5,-0.5)[place] {};
\draw [nright] (1) to [bend right=10] (2);
\draw [nright] (2) to [bend right=10] (1);
\draw [nright] (3) to [bend right=10] (4);
\draw [nright] (4) to [bend right=10] (3);
\draw [right] (2) to (3);
\draw [right] (4) to (1);
\draw [-] (-0.5,0.55) arc (360:0:3pt);
\endtikzpicture}
& $[$& 1 & 1& 0& 0;& $-1$& 0& 1& 0;& 0& 0& 0& 1;& 1& 0& $-2$& $\left. 0\right]$ & $\left(4, 0, 0, 0\right)$\\
& & $[$& -1 & 1& 0& 0;& 1& 0& $-1$& 0;& 0& 0& 0& 1;& 1& 0& $-2$& $\left. 0\right]$ & $\left(3, 1, 0, 0\right)$\\
& & $[$& 1 & 1& 0& 0;& $-1$& 0& 1& 0;& 0& 0& 0& 1;& 1& 0& $-1$& $\left. 0\right]$ & $\left(3, 0, 1, 0\right)$\\
& & $[$& -1 & 1& 0& 0;& $-1$& 0& $-1$& 0;& 0& 0& 0& 1;& 1& 0& $-1$& $\left. 0\right]$ & $\left(2, 2, 0, 0\right)$\\
& & $[$& -1 & 1& 0& 0;& 1& 0& $-1$& 0;& 0& 0& 0& 1;& 1& 0& $-1$& $\left. 0\right]$ & $\left(2, 1, 1, 0\right)$\\
& & $[$& 1 & 1& 0& 0;& $-1$& 0& 0& 0;& 0& 0& 0& 1;& 1& 0& 0& $\left. 0\right]$ & $\left(2, 0, 2, 0\right)$\\
& & $[$& 1 & 1& 0& 0;& $-1$& 0& 0& 0;& 0& 0& 0& 1;& 1& 0& $-1$& $\left. 0\right]$ & $\left(2, 0, 0, 2\right)$\\
& & $[$& -1 & 1& 0& 0;& 1& 0& 0& 0;& 0& 0& 0& 1;& 1& 0& 0& $\left. 0\right]$ & $\left(1, 1, 2, 0\right)$\\
& & $[$& -1 & 1& 0& 0;& 1& 0& 0& 0;& 0& 0& 0& 1;& 1& 0& $-1$& $\left. 0\right]$ & $\left(1, 1, 0, 2\right)$\\
& & $[$& 1 & 1& 0& 0;& 0& 0& 0& 0;& 0& 0& 0& 1;& 1& 0& 0& $\left. 0\right]$ & $\left(1, 0, 3, 0\right)$\\
& & $[$& 1 & 1& 0& 0;& 0& 0& 0& 0;& 0& 0& 0& 1;& 1& 0& $-1$& $\left. 0\right]$ & $\left(1, 0, 1, 2\right)$\\
& & $[$& 0 & 1& 0& 0;& $-1$& 0& $-1$& 0;& 0& 0& 0& 1;& 1& 0& 1& $\left. 0\right]$ & $\left(0, 0, 4, 0\right)$\\
& & $[$& 0 & 1& 0& 0;& $-1$& 0& 0& 0;& 0& 0& 0& 1;& 1& 0& 0& $\left. 0\right]$ & $\left(0, 0, 2, 2\right)$\\
& & $[$& 0 & 1& 0& 0;& $-1$& 0& 0& 0;& 0& 0& 0& 1;& 1& 0& $-1$& $\left. 0\right]$ & $\left(0, 0, 0, 4\right)$\\
\hline

\multirow{5}{*}{$\R_4$}
&   \multirow{5}{*}{\tikzpicture \phantom{\node (5) at (-0.25,0.28)[place]{};}
\node (1) at (-0.5,0.5)[place] {};
\node (2) at (0.5,0.5)[place] {};
\node (3) at (0.5,-0.5)[place] {};
\node (4) at (-0.5,-0.5)[place] {};
\draw [nright] (1) to [bend right=10] (2);
\draw [nright] (2) to [bend right=10] (1);
\draw [nright] (3) to [bend right=10] (4);
\draw [nright] (4) to [bend right=10] (3);
\draw [right] (2) to (3);
\draw [right] (4) to (2);
\draw [-] (-0.5,0.55) arc (360:0:3pt);
\endtikzpicture}
& $[$& 1 & $-1$& 0& 0;& 1& 0& $-1$& 0;& 0& 0& 0& 1;& 0& 1& $-2$& $\left. 0\right]$ & $\left(4, 0, 0, 0\right)$\\
& & $[$& -1 & 1& 0& 0;& 1& 0& $-1$& 0;& 0& 0& 0& 1;& 0& 1& $-2$& $\left. 0\right]$ & $\left(3, 1, 0, 0\right)$\\
& & $[$& 2 & $-1$& 0& 0;& 1& 0& $-1$& 0;& 0& 0& 0& 1;& 0& 1& $-2$& $\left. 0\right]$ & $\left(3, 0, 1, 0\right)$\\
& & $[$& -1 & $-1$& 0& 0;& 1& 0& $-1$& 0;& 0& 0& 0& 1;& 0& 1& $-1$& $\left. 0\right]$ & $\left(2, 2, 0, 0\right)$\\
& & $[$& -1 & 1& 0& 0;& 1& 0& $-1$& 0;& 0& 0& 0& 1;& 0& 1& $-1$& $\left. 0\right]$ & $\left(2, 1, 1, 0\right)$\\
& & $[$& 1 & $-1$& 0& 0;& 1& 0& $-1$& 0;& 0& 0& 0& 1;& 0& 1& $-1$& $\left. 0\right]$ & $\left(2, 0, 2, 0\right)$\\
& & $[$& 1 & $-1$& 0& 0;& 1& 0& 0& 0;& 0& 0& 0& 1;& 0& 1& $-1$& $\left. 0\right]$ & $\left(2, 0, 0, 2\right)$\\
& & $[$& -1 & $-1$& 0& 0;& 1& 0& $-2$& 0;& 0& 0& 0& 1;& 0& 1& 2& $\left. 0\right]$ & $\left(1, 1, 2, 0\right)$\\
& & $[$& -1 & 1& 0& 0;& 1& 0& 0& 0;& 0& 0& 0& 1;& 0& 1& $-1$& $\left. 0\right]$ & $\left(1, 1, 0, 2\right)$\\
& & $[$& 1 & $-1$& 0& 0;& 1& 0& 1& 0;& 0& 0& 0& 1;& 0& 1& 1& $\left. 0\right]$ & $\left(1, 0, 3, 0\right)$\\
& & $[$& 1 & 0& 0& 0;& 1& 0& 0& 0;& 0& 0& 0& 1;& 0& 1& $-1$& $\left. 0\right]$ & $\left(1, 0, 1, 2\right)$\\
& & $[$& 0 & 0& 0& 0;& 1& 0& 0& 0;& 0& 0& 0& 1;& 0& 1& 0& $\left. 0\right]$ & $\left(0, 0, 4, 0\right)$\\
& & $[$& 0 & $-1$& 0& 0;& 1& 0& 0& 0;& 0& 0& 0& 1;& 0& 1& 0& $\left. 0\right]$ & $\left(0, 0, 2, 2\right)$\\
& & $[$& 0 & $-1$& 0& 0;& 1& 0& 0& 0;& 0& 0& 0& 1;& 0& 1& $-1$& $\left. 0\right]$ & $\left(0, 0, 0, 4\right)$\\
\hline

\multirow{5}{*}{$\R_5$}
&   \multirow{5}{*}{\tikzpicture \phantom{\node (5) at (-0.25,0.28)[place]{};}
\node (1) at (-0.5,0.5)[place] {};
\node (2) at (0.5,0.5)[place] {};
\node (3) at (0.5,-0.5)[place] {};
\node (4) at (-0.5,-0.5)[place] {};
\draw [nright] (1) to [bend right=10] (2);
\draw [nright] (2) to [bend right=10] (1);
\draw [nright] (3) to [bend right=10] (4);
\draw [nright] (4) to [bend right=10] (3);
\draw [right] (1) to (4);
\draw [right] (3) to (1);
\draw [-] (-0.5,0.55) arc (360:0:3pt);
\endtikzpicture}
& $[$& 2 & $-1$& 0& $-1;$& 1& 0& 0& 0;& 1& 0& 0& $-2;$& 0& 0& 1& $\left. 0\right]$ & $\left(4, 0, 0, 0\right)$\\
& & $[$& -1 & 1& 0& $-1;$& 1& 0& 0& 0;& 1& 0& 0& $-1;$& 0& 0& 1& $\left. 0\right]$ & $\left(3, 1, 0, 0\right)$\\
& & $[$& 1 & 0& 0& $-1;$& 1& 0& 0& 0;& 1& 0& 0& $-2;$& 0& 0& 1& $\left. 0\right]$ & $\left(3, 0, 1, 0\right)$\\
& & $[$& -1 & $-1$& 0& $-1;$& 1& 0& 0& 0;& 1& 0& 0& $-1;$& 0& 0& 1& $\left. 0\right]$ & $\left(2, 2, 0, 0\right)$\\
& & $[$& -1 & $-1$& 0& $-2;$& 1& 0& 0& 0;& 1& 0& 0& 0;& 0& 0& 1& $\left. 0\right]$ & $\left(2, 1, 1, 0\right)$\\
& & $[$& 1 & $-1$& 0& 0;& 1& 0& 0& 0;& 1& 0& 0& 0;& 0& 0& 1& $\left. 0\right]$ & $\left(2, 0, 2, 0\right)$\\
& & $[$& 1 & $-1$& 0& $-1;$& 1& 0& 0& 0;& 1& 0& 0& $-2;$& 0& 0& 1& $\left. 0\right]$ & $\left(2, 0, 0, 2\right)$\\
& & $[$& -1 & 0& 0& $-1;$& 1& 0& 0& 0;& 1& 0& 0& 1;& 0& 0& 1& $\left. 0\right]$ & $\left(1, 1, 2, 0\right)$\\
& & $[$& -1 & $-1$& 0& $-2;$& 1& 0& 0& 0;& 1& 0& 0& 1;& 0& 0& 1& $\left. 0\right]$ & $\left(1, 1, 0, 2\right)$\\
& & $[$& 1 & 0& 0& 0;& 1& 0& 0& 0;& 1& 0& 0& 0;& 0& 0& 1& $\left. 0\right]$ & $\left(1, 0, 3, 0\right)$\\
& & $[$& 1 & $-1$& 0& 1;& 1& 0& 0& 0;& 1& 0& 0& 0;& 0& 0& 1& $\left. 0\right]$ & $\left(1, 0, 1, 2\right)$\\
& & $[$& 0 & 0& 0& 0;& 1& 0& 0& 0;& 1& 0& 0& 0;& 0& 0& 1& $\left. 0\right]$ & $\left(0, 0, 4, 0\right)$\\
& & $[$& 0 & $-1$& 0& 0;& 1& 0& 0& 0;& 1& 0& 0& 0;& 0& 0& 1& $\left. 0\right]$ & $\left(0, 0, 2, 2\right)$\\
& & $[$& 0 & $-1$& 0& 0;& 1& 0& 0& 0;& 1& 0& 0& $-1;$& 0& 0& 1& $\left. 0\right]$ & $\left(0, 0, 0, 4\right)$\\
\hline

\multirow{5}{*}{$\R_6$}
&   \multirow{5}{*}{\tikzpicture \phantom{\node (5) at (-0.25,0.28)[place]{};}
\node (1) at (-0.5,0.5)[place] {};
\node (2) at (0.5,0.5)[place] {};
\node (3) at (0.5,-0.5)[place] {};
\node (4) at (-0.5,-0.5)[place] {};
\draw [nright] (1) to [bend right=10] (2);
\draw [nright] (2) to [bend right=10] (1);
\draw [right] (2) to (3);
\draw [right] (3) to (4);
\draw [right] (4) to (2);
\draw [-] (-0.5,0.55) arc (360:0:3pt);
\draw [-] (0.7,-0.55) arc (360:0:3pt);
\endtikzpicture}
& $[$& 1 & $-1$& 0& 0;& 1& 0& 1& 0;& 0& 0& 2& 1;& 0& 1& 0& $\left. 0\right]$ & $\left(4, 0, 0, 0\right)$\\
& & $[$& -1 & 1& 0& 0;& 1& 0& 1& 0;& 0& 0& 2& 1;& 0& 1& 0& $\left. 0\right]$ & $\left(3, 1, 0, 0\right)$\\
& & $[$& 0 & $-1$& 0& 0;& 1& 0& $-1$& 0;& 0& 0& 2& 1;& 0& 1& 0& $\left. 0\right]$ & $\left(3, 0, 1, 0\right)$\\
& & $[$& -1 & $-1$& 0& 0;& 1& 0& $-1$& 0;& 0& 0& 0& 1;& 0& 1& 0& $\left. 0\right]$ & $\left(2, 2, 0, 0\right)$\\
& & $[$& -1 & 1& 0& 0;& 1& 0& 0& 0;& 0& 0& 1& 1;& 0& 1& 0& $\left. 0\right]$ & $\left(2, 1, 1, 0\right)$\\
& & $[$& 0 & $-1$& 0& 0;& 1& 0& $-1$& 0;& 0& 0& 1& 1;& 0& 1& 0& $\left. 0\right]$ & $\left(2, 0, 2, 0\right)$\\
& & $[$& 1 & $-1$& 0& 0;& 1& 0& 1& 0;& 0& 0& 1& 1;& 0& 1& 0& $\left. 0\right]$ & $\left(2, 0, 0, 2\right)$\\
& & $[$& -1 & 0& 0& 0;& 1& 0& 0& 0;& 0& 0& 1& 1;& 0& 1& 0& $\left. 0\right]$ & $\left(1, 1, 2, 0\right)$\\
& & $[$& -1 & 1& 0& 0;& 1& 0& 1& 0;& 0& 0& 1& 1;& 0& 1& 0& $\left. 0\right]$ & $\left(1, 1, 0, 2\right)$\\
& & $[$& 0 & 0& 0& 0;& 1& 0& 0& 0;& 0& 0& 1& 1;& 0& 1& 0& $\left. 0\right]$ & $\left(1, 0, 3, 0\right)$\\
& & $[$& 0 & $-1$& 0& 0;& 1& 0& 0& 0;& 0& 0& 1& 1;& 0& 1& 0& $\left. 0\right]$ & $\left(1, 0, 1, 2\right)$\\
& & $[$& 0 & 0& 0& 0;& 1& 0& 0& 0;& 0& 0& 0& 1;& 0& 1& 0& $\left. 0\right]$ & $\left(0, 0, 4, 0\right)$\\
& & $[$& 0 & $-1$& 0& 0;& 1& 0& 0& 0;& 0& 0& 0& 1;& 0& 1& 0& $\left. 0\right]$ & $\left(0, 0, 2, 2\right)$\\
& & $[$& $-\frac{1}{2}$ & $-2$& 0& 0;& 1& 0& $-1$& 0;& 0& 0& $\frac{1}{2}$& 1;& 0& 1& 0& $\left. 0\right]$ & $\left(0, 0, 0, 4\right)$\\
\hline

\multirow{5}{*}{$\R_7$}
&   \multirow{5}{*}{\tikzpicture \phantom{\node (5) at (-0.25,0.28)[place]{};}
\node (1) at (-0.5,0.5)[place] {};
\node (2) at (0.5,0.5)[place] {};
\node (3) at (0.5,-0.5)[place] {};
\node (4) at (-0.5,-0.5)[place] {};
\draw [nright] (3) to [bend right=10] (4);
\draw [nright] (4) to [bend right=10] (3);
\draw [right] (1) to (2);
\draw [right] (2) to (3);
\draw [right] (4) to (1);
\draw [-] (-0.5,0.55) arc (360:0:3pt);
\draw [-] (0.7,0.55) arc (360:0:3pt);
\endtikzpicture}
& $[$& 1 & 1& 0& 0;& 0& 1& 1& 0;& 0& 0& 0& 1;& 1& 0& $-2$& $\left. 0\right]$ & $\left(4, 0, 0, 0\right)$\\
& & $[$& -1 & 1& 0& 0;& 0& 2& 1& 0;& 0& 0& 0& 1;& 1& 0& $-1$& $\left. 0\right]$ & $\left(3, 1, 0, 0\right)$\\
& & $[$& 1 & 1& 0& 0;& 0& 1& 1& 0;& 0& 0& 0& 1;& 1& 0& $-1$& $\left. 0\right]$ & $\left(3, 0, 1, 0\right)$\\
& & $[$& -1 & 1& 0& 0;& 0& $-1$& 1& 0;& 0& 0& 0& 1;& $-1$& 0& $-1$& $\left. 0\right]$ & $\left(2, 2, 0, 0\right)$\\
& & $[$& -1 & 1& 0& 0;& 0& 2& 1& 0;& 0& 0& 0& 1;& 2& 0& 1& $\left. 0\right]$ & $\left(2, 1, 1, 0\right)$\\
& & $[$& 1 & 1& 0& 0;& 0& 1& 1& 0;& 0& 0& 0& 1;& 0& 0& 0& $\left. 0\right]$ & $\left(2, 0, 2, 0\right)$\\
& & $[$& 1 & 1& 0& 0;& 0& 1& 1& 0;& 0& 0& 0& 1;& 0& 0& $-1$& $\left. 0\right]$ & $\left(2, 0, 0, 2\right)$\\
& & $[$& -1 & 1& 0& 0;& 0& 1& 1& 0;& 0& 0& 0& 1;& 0& 0& 0& $\left. 0\right]$ & $\left(1, 1, 2, 0\right)$\\
& & $[$& -1 & 1& 0& 0;& 0& 1& 1& 0;& 0& 0& 0& 1;& $-1$& 0& $-2$& $\left. 0\right]$ & $\left(1, 1, 0, 2\right)$\\
& & $[$& 0 & 1& 0& 0;& 0& 1& 1& 0;& 0& 0& 0& 1;& 0& 0& 0& $\left. 0\right]$ & $\left(1, 0, 3, 0\right)$\\
& & $[$& 0 & 1& 0& 0;& 0& 1& 1& 0;& 0& 0& 0& 1;& 0& 0& $-1$& $\left. 0\right]$ & $\left(1, 0, 1, 2\right)$\\
& & $[$& -1 & 1& 0& 0;& 0& 1& 1& 0;& 0& 0& 0& 1;& $-1$& 0& $-1$& $\left. 0\right]$ & $\left(0, 0, 4, 0\right)$\\
& & $[$& -1 & 1& 0& 0;& 0& 1& 1& 0;& 0& 0& 0& 1;& $-2$& 0& $-2$& $\left. 0\right]$ & $\left(0, 0, 2, 2\right)$\\
& & $[$& 0 & 1& 0& 0;& 0& 0& 1& 0;& 0& 0& 0& 1;& $-1$& 0& $-2$& $\left. 0\right]$ & $\left(0, 0, 0, 4\right)$\\
\hline
\end{longtable}

\begin{longtable}{rrlrrrrrrrrrrrrrrrrr}
\caption{Inertially arbitrary patterns with a proper $2$-cycle.}\label{IAPT}\\
\hline
\hline
\multicolumn{2}{c}{Pattern}  & \multicolumn{17}{c}{Matrix} & Inertia \\
\hline
\hline

\multirow{5}{*}{$\B_1$}
&   \multirow{5}{*}{\tikzpicture \phantom{\node (5) at (-0.25,0.28)[place]{};}
\node (1) at (-0.5,0.5)[place] {};
\node (2) at (0.5,0.5)[place] {};
\node (3) at (0.5,-0.5)[place] {};
\node (4) at (-0.5,-0.5)[place] {};
\draw [nright] (1) to [bend right=10] (3);
\draw [nright] (3) to [bend right=10] (1);
\draw [right] (1) to (2);
\draw [right] (2) to (3);
\draw [right] (3) to (4);
\draw [right] (4) to (1);
\draw [-] (-0.5,0.55) arc (360:0:3pt);
\endtikzpicture}
& $[$& 2 & $-1$& 2& 0;& 0& 0& 1& 0;& $-2$& 0& 0& 1;& 1& 0& 0& $\left. 0\right]$ & $\left(4, 0, 0\right)$\\
& & $[$& -1 & $-1$& 0& 0;& 0& 0& 1& 0;& 2& 0& 0& $-1;$& 1& 0& 0& $\left. 0\right]$ & $\left(3, 1, 0\right)$\\
& & $[$& 1 & $-1$& 2& 0;& 0& 0& 1& 0;& $-1$& 0& 0& 0;& 1& 0& 0& $\left. 0\right]$ & $\left(3, 0, 1\right)$\\
& & $[$& -1 & $-1$& $-1$& 0;& 0& 0& 1& 0;& $-1$& 0& 0& 1;& 1& 0& 0& $\left. 0\right]$ & $\left(2, 2, 0\right)$\\
& & $[$& -1 & $-1$& 0& 0;& 0& 0& 1& 0;& 1& 0& 0& 0;& 1& 0& 0& $\left. 0\right]$ & $\left(2, 1, 1\right)$\\
& & $[$& 1 & 0& $-1$& 0;& 0& 0& 1& 0;& 1& 0& 0& 0;& 1& 0& 0& $\left. 0\right]$ & $\left(2, 0, 2\right)$\\
& & $[$& -1 & $-1$& 0& 0;& 0& 0& 1& 0;& 1& 0& 0& $-1;$& 1& 0& 0& $\left. 0\right]$ & $\left(1, 1, 2\right)$\\
& & $[$& 1 & $-1$& $-1$& 0;& 0& 0& 1& 0;& 0& 0& 0& 0;& 1& 0& 0& $\left. 0\right]$ & $\left(1, 0, 3\right)$\\
& & $[$& 0 & $-1$& $-1$& 0;& 0& 0& 1& 0;& 0& 0& 0& 0;& 1& 0& 0& $\left. 0\right]$ & $\left(0, 0, 4\right)$\\
\hline

\multirow{5}{*}{$\B_2$}
&   \multirow{5}{*}{\tikzpicture \phantom{\node (5) at (-0.25,0.28)[place]{};}
\node (1) at (-0.5,0.5)[place] {};
\node (2) at (0.5,0.5)[place] {};
\node (3) at (0.5,-0.5)[place] {};
\node (4) at (-0.5,-0.5)[place] {};
\draw [nright] (1) to [bend right=10] (2);
\draw [nright] (2) to [bend right=10] (1);
\draw [right] (2) to (4);
\draw [right] (3) to (2);
\draw [right] (4) to (3);
\draw [right] (4) to (1);
\draw [-] (-0.5,0.55) arc (360:0:3pt);
\endtikzpicture}
& $[$& 8 & 1& 0& 0;& $-24$& 0& 0& 1;& 0& 2& 0& 0;& 30& 0& 1& $\left. 0\right]$ & $\left(4, 0, 0\right)$\\
& & $[$& -1 & 1& 0& 0;& 1& 0& 0& 1;& 0& 1& 0& 0;& $-2$& 0& 1& $\left. 0\right]$ & $\left(3, 1, 0\right)$\\
& & $[$& 1 & 1& 0& 0;& $-2$& 0& 0& 1;& 0& 0& 0& 0;& 1& 0& 1& $\left. 0\right]$ & $\left(3, 0, 1\right)$\\
& & $[$& -1 & 1& 0& 0;& $-1$& 0& 0& 1;& 0& $-1$& 0& 0;& $-1$& 0& 1& $\left. 0\right]$ & $\left(2, 2, 0\right)$\\
& & $[$& -1 & 1& 0& 0;& $-1$& 0& 0& 1;& 0& 0& 0& 0;& $-2$& 0& 1& $\left. 0\right]$ & $\left(2, 1, 1\right)$\\
& & $[$& 1 & 1& 0& 0;& $-1$& 0& 0& 1;& 0& 0& 0& 0;& 0& 0& 1& $\left. 0\right]$ & $\left(2, 0, 2\right)$\\
& & $[$& -1 & 1& 0& 0;& 0& 0& 0& 1;& 0& 1& 0& 0;& $-2$& 0& 1& $\left. 0\right]$ & $\left(1, 1, 2\right)$\\
& & $[$& 1 & 1& 0& 0;& $-1$& 0& 0& 1;& 0& 0& 0& 0;& 1& 0& 1& $\left. 0\right]$ & $\left(1, 0, 3\right)$\\
& & $[$& 0 & 1& 0& 0;& $-1$& 0& 0& 1;& 0& $-1$& 0& 0;& 1& 0& 1& $\left. 0\right]$ & $\left(0, 0, 4\right)$\\
\hline

\multirow{5}{*}{$\B_3$}
&   \multirow{5}{*}{\tikzpicture \phantom{\node (5) at (-0.25,0.28)[place]{};}
\node (1) at (-0.5,0.5)[place] {};
\node (2) at (0.5,0.5)[place] {};
\node (3) at (0.5,-0.5)[place] {};
\node (4) at (-0.5,-0.5)[place] {};
\draw [nright] (2) to [bend right=10] (4);
\draw [nright] (4) to [bend right=10] (2);
\draw [right] (1) to (2);
\draw [right] (2) to (3);
\draw [right] (3) to (4);
\draw [right] (4) to (1);
\draw [-] (-0.5,0.55) arc (360:0:3pt);
\endtikzpicture}
& $[$& 2 & $-1$& 0& 0;& 0& 0& $-1$& 2;& 0& 0& 0& 1;& 2& $-2$& 0& $\left. 0\right]$ & $\left(4, 0, 0\right)$\\
& & $[$& -1 & $-1$& 0& 0;& 0& 0& $-1$& 1;& 0& 0& 0& 1;& 2& 0& 0& $\left. 0\right]$ & $\left(3, 1, 0\right)$\\
& & $[$& 1 & $-1$& 0& 0;& 0& 0& $-1$& 2;& 0& 0& 0& 1;& 1& $-1$& 0& $\left. 0\right]$ & $\left(3, 0, 1\right)$\\
& & $[$& -1 & $-1$& 0& 0;& 0& 0& $-1$& $-1;$& 0& 0& 0& 1;& $-1$& 0& 0& $\left. 0\right]$ & $\left(2, 2, 0\right)$\\
& & $[$& -1 & $-1$& 0& 0;& 0& 0& $-1$& 0;& 0& 0& 0& 1;& 1& 1& 0& $\left. 0\right]$ & $\left(2, 1, 1\right)$\\
& & $[$& 1 & $-1$& 0& 0;& 0& 0& 0& $-1;$& 0& 0& 0& 1;& $-1$& 1& 0& $\left. 0\right]$ & $\left(2, 0, 2\right)$\\
& & $[$& -1 & $-1$& 0& 0;& 0& 0& $-1$& 0;& 0& 0& 0& 1;& 2& 1& 0& $\left. 0\right]$ & $\left(1, 1, 2\right)$\\
& & $[$& 1 & $-1$& 0& 0;& 0& 0& $-1$& $-1;$& 0& 0& 0& 1;& 0& 0& 0& $\left. 0\right]$ & $\left(1, 0, 3\right)$\\
& & $[$& 0 & $-1$& 0& 0;& 0& 0& $-1$& $-1;$& 0& 0& 0& 1;& 0& 0& 0& $\left. 0\right]$ & $\left(0, 0, 4\right)$\\
\hline

\multirow{5}{*}{$\B_4$}
&   \multirow{5}{*}{\tikzpicture \phantom{\node (5) at (-0.25,0.28)[place]{};}
\node (1) at (-0.5,0.5)[place] {};
\node (2) at (0.5,0.5)[place] {};
\node (3) at (0.5,-0.5)[place] {};
\node (4) at (-0.5,-0.5)[place] {};
\draw [nright] (2) to [bend right=10] (3);
\draw [nright] (3) to [bend right=10] (2);
\draw [right] (1) to (2);
\draw [right] (2) to (4);
\draw [right] (4) to (3);
\draw [right] (4) to (1);
\draw [-] (-0.5,0.55) arc (360:0:3pt);
\endtikzpicture}
& $[$& 8 & $-162$& 0& 0;& 0& 0& $-24$& 1;& 0& 1& 0& 0;& 1& 0& 2& $\left. 0\right]$ & $\left(4, 0, 0\right)$\\
& & $[$& -1 & $-2$& 0& 0;& 0& 0& $-1$& 1;& 0& 1& 0& 0;& 1& 0& 1& $\left. 0\right]$ & $\left(3, 1, 0\right)$\\
& & $[$& 1 & $-1$& 0& 0;& 0& 0& $-2$& 1;& 0& 1& 0& 0;& 1& 0& 0& $\left. 0\right]$ & $\left(3, 0, 1\right)$\\
& & $[$& -1 & $-1$& 0& 0;& 0& 0& $-1$& 1;& 0& 1& 0& 0;& 1& 0& $-1$& $\left. 0\right]$ & $\left(2, 2, 0\right)$\\
& & $[$& -1 & $-1$& 0& 0;& 0& 0& $-1$& 1;& 0& 1& 0& 0;& 1& 0& 0& $\left. 0\right]$ & $\left(2, 1, 1\right)$\\
& & $[$& 1 & $-1$& 0& 0;& 0& 0& $-1$& 1;& 0& 1& 0& 0;& 1& 0& 0& $\left. 0\right]$ & $\left(2, 0, 2\right)$\\
& & $[$& -1 & $-1$& 0& 0;& 0& 0& 1& 1;& 0& 1& 0& 0;& 1& 0& 0& $\left. 0\right]$ & $\left(1, 1, 2\right)$\\
& & $[$& 1 & 0& 0& 0;& 0& 0& $-1$& 1;& 0& 1& 0& 0;& 1& 0& 0& $\left. 0\right]$ & $\left(1, 0, 3\right)$\\
& & $[$& 0 & $-1$& 0& 0;& 0& 0& $-1$& 1;& 0& 1& 0& 0;& 1& 0& 1& $\left. 0\right]$ & $\left(0, 0, 4\right)$\\
\hline

\multirow{5}{*}{$\B_5$}
&   \multirow{5}{*}{\tikzpicture \phantom{\node (5) at (-0.25,0.28)[place]{};}
\node (1) at (-0.5,0.5)[place] {};
\node (2) at (0.5,0.5)[place] {};
\node (3) at (0.5,-0.5)[place] {};
\node (4) at (-0.5,-0.5)[place] {};
\draw [nright] (4) to [bend right=10] (2);
\draw [nright] (2) to [bend right=10] (4);
\draw [right] (1) to (2);
\draw [right] (3) to (2);
\draw [right] (4) to (3);
\draw [right] (4) to (1);
\draw [-] (-0.5,0.55) arc (360:0:3pt);
\endtikzpicture}
& $[$& 8 & $-162$& 0& 0;& 0& 0& 0& 1;& 0& 2& 0& 0;& 1& $-24$& 1& $\left. 0\right]$ & $\left(4, 0, 0\right)$\\
& & $[$& -1 & $-2$& 0& 0;& 0& 0& 0& 1;& 0& 1& 0& 0;& 1& $-1$& 1& $\left. 0\right]$ & $\left(3, 1, 0\right)$\\
& & $[$& 1 & $-1$& 0& 0;& 0& 0& 0& 1;& 0& 0& 0& 0;& 1& $-2$& 1& $\left. 0\right]$ & $\left(3, 0, 1\right)$\\
& & $[$& -1 & $-1$& 0& 0;& 0& 0& 0& 1;& 0& $-1$& 0& 0;& 1& $-1$& 1& $\left. 0\right]$ & $\left(2, 2, 0\right)$\\
& & $[$& -1 & $-1$& 0& 0;& 0& 0& 0& 1;& 0& 0& 0& 0;& 1& $-1$& 1& $\left. 0\right]$ & $\left(2, 1, 1\right)$\\
& & $[$& 1 & $-1$& 0& 0;& 0& 0& 0& 1;& 0& 0& 0& 0;& 1& $-1$& 1& $\left. 0\right]$ & $\left(2, 0, 2\right)$\\
& & $[$& -1 & $-1$& 0& 0;& 0& 0& 0& 1;& 0& 0& 0& 0;& 1& 1& 1& $\left. 0\right]$ & $\left(1, 1, 2\right)$\\
& & $[$& 1 & 0& 0& 0;& 0& 0& 0& 1;& 0& 0& 0& 0;& 1& $-1$& 1& $\left. 0\right]$ & $\left(1, 0, 3\right)$\\
& & $[$& 0 & $-1$& 0& 0;& 0& 0& 0& 1;& 0& 1& 0& 0;& 1& $-1$& 1& $\left. 0\right]$ & $\left(0, 0, 4\right)$\\
\hline

\multirow{5}{*}{$\B_6$}
&   \multirow{5}{*}{\tikzpicture \phantom{\node (5) at (-0.25,0.28)[place]{};}
\node (1) at (-0.5,0.5)[place] {};
\node (2) at (0.5,0.5)[place] {};
\node (3) at (0.5,-0.5)[place] {};
\node (4) at (-0.5,-0.5)[place] {};
\draw [nright] (1) to [bend right=10] (2);
\draw [nright] (2) to [bend right=10] (1);
\draw [nright] (3) to [bend right=10] (2);
\draw [nright] (2) to [bend right=10] (3);
\draw [right] (3) to (4);
\draw [right] (4) to (2);
\draw [-] (-0.5,0.55) arc (360:0:3pt);
\endtikzpicture}
& $[$& 1 & $-2$& 0& 0;& 1& 0& 1& 0;& 0& $-1$& 0& 1;& 0& 1& 0& $\left. 0\right]$ & $\left(4, 0, 0\right)$\\
& & $[$& 1 & $-1$& 0& 0;& 1& 0& 1& 0;& 0& $-1$& 0& $-1;$& 0& 1& 0& $\left. 0\right]$ & $\left(3, 1, 0\right)$\\
& & $[$& 1 & $-1$& 0& 0;& 1& 0& 1& 0;& 0& $-1$& 0& 0;& 0& 1& 0& $\left. 0\right]$ & $\left(3, 0, 1\right)$\\
& & $[$& -1 & $-1$& 0& 0;& 1& 0& 1& 0;& 0& $-1$& 0& $-1;$& 0& 1& 0& $\left. 0\right]$ & $\left(2, 2, 0\right)$\\
& & $[$& -1 & 1& 0& 0;& 1& 0& 1& 0;& 0& $-1$& 0& 0;& 0& 1& 0& $\left. 0\right]$ & $\left(2, 1, 1\right)$\\
& & $[$& 1 & $-1$& 0& 0;& 1& 0& 1& 0;& 0& 0& 0& 0;& 0& 1& 0& $\left. 0\right]$ & $\left(2, 0, 2\right)$\\
& & $[$& -1 & 1& 0& 0;& 1& 0& 1& 0;& 0& 0& 0& 0;& 0& 1& 0& $\left. 0\right]$ & $\left(1, 1, 2\right)$\\
& & $[$& 1 & 0& 0& 0;& 1& 0& 1& 0;& 0& $-1$& 0& 0;& 0& 1& 0& $\left. 0\right]$ & $\left(1, 0, 3\right)$\\
& & $[$& 0 & $-1$& 0& 0;& 1& 0& 1& 0;& 0& $-1$& 0& 0;& 0& 1& 0& $\left. 0\right]$ & $\left(0, 0, 4\right)$\\
\hline

\multirow{5}{*}{$\B_7$}
&   \multirow{5}{*}{\tikzpicture \phantom{\node (5) at (-0.25,0.28)[place]{};}
\node (1) at (-0.5,0.5)[place] {};
\node (2) at (0.5,0.5)[place] {};
\node (3) at (0.5,-0.5)[place] {};
\node (4) at (-0.5,-0.5)[place] {};
\draw [nright] (1) to [bend right=10] (2);
\draw [nright] (2) to [bend right=10] (1);
\draw [right] (2) to (3);
\draw [right] (3) to (4);
\draw [right] (4) to (2);
\draw [-] (-0.5,0.55) arc (360:0:3pt);
\draw [-] (0.7,0.55) arc (360:0:3pt);
\endtikzpicture}
& $[$& 1 & $-2$& 0& 0;& 1& 1& 1& 0;& 0& 0& 0& 1;& 0& 1& 0& $\left. 0\right]$ & $\left(4, 0, 0\right)$\\
& & $[$& 1 & $-1$& 0& 0;& 1& $-1$& $-1$& 0;& 0& 0& 0& 1;& 0& 1& 0& $\left. 0\right]$ & $\left(3, 1, 0\right)$\\
& & $[$& 0 & $-1$& 0& 0;& 1& 2& 1& 0;& 0& 0& 0& 1;& 0& 1& 0& $\left. 0\right]$ & $\left(3, 0, 1\right)$\\
& & $[$& -1 & $-1$& 0& 0;& 1& $-1$& $-1$& 0;& 0& 0& 0& 1;& 0& 1& 0& $\left. 0\right]$ & $\left(2, 2, 0\right)$\\
& & $[$& 0 & $-1$& 0& 0;& 1& $-1$& $-2$& 0;& 0& 0& 0& 1;& 0& 1& 0& $\left. 0\right]$ & $\left(2, 1, 1\right)$\\
& & $[$& 0 & $-1$& 0& 0;& 1& 1& 0& 0;& 0& 0& 0& 1;& 0& 1& 0& $\left. 0\right]$ & $\left(2, 0, 2\right)$\\
& & $[$& -1 & $-1$& 0& 0;& 1& 2& 0& 0;& 0& 0& 0& 1;& 0& 1& 0& $\left. 0\right]$ & $\left(1, 1, 2\right)$\\
& & $[$& -1 & $-2$& 0& 0;& 1& 2& 0& 0;& 0& 0& 0& 1;& 0& 1& 0& $\left. 0\right]$ & $\left(1, 0, 3\right)$\\
& & $[$& -1 & $-1$& 0& 0;& 1& 1& 0& 0;& 0& 0& 0& 1;& 0& 1& 0& $\left. 0\right]$ & $\left(0, 0, 4\right)$\\
\hline

\multirow{5}{*}{$\B_8$}
&   \multirow{5}{*}{\tikzpicture \phantom{\node (5) at (-0.25,0.28)[place]{};}
\node (1) at (-0.5,0.5)[place] {};
\node (2) at (0.5,0.5)[place] {};
\node (3) at (0.5,-0.5)[place] {};
\node (4) at (-0.5,-0.5)[place] {};
\draw [nright] (3) to [bend right=10] (4);
\draw [nright] (4) to [bend right=10] (3);
\draw [right] (1) to (2);
\draw [right] (2) to (3);
\draw [right] (3) to (1);
\draw [-] (-0.5,0.55) arc (360:0:3pt);
\draw [-] (0.7,0.55) arc (360:0:3pt);
\endtikzpicture}
& $[$& 1 & $-1$& 0& 0;& 0& 1& 1& 0;& 1& 0& 0& $-2;$& 0& 0& 1& $\left. 0\right]$ & $\left(4, 0, 0\right)$\\
& & $[$& -1 & $-1$& 0& 0;& 0& 1& 1& 0;& 1& 0& 0& $-1;$& 0& 0& 1& $\left. 0\right]$ & $\left(3, 1, 0\right)$\\
& & $[$& 0 & $-1$& 0& 0;& 0& 1& 1& 0;& 1& 0& 0& $-2;$& 0& 0& 1& $\left. 0\right]$ & $\left(3, 0, 1\right)$\\
& & $[$& -1 & $-1$& 0& 0;& 0& $-1$& 1& 0;& 1& 0& 0& $-1;$& 0& 0& 1& $\left. 0\right]$ & $\left(2, 2, 0\right)$\\
& & $[$& -1 & $-1$& 0& 0;& 0& 0& 1& 0;& 1& 0& 0& $-1;$& 0& 0& 1& $\left. 0\right]$ & $\left(2, 1, 1\right)$\\
& & $[$& 0 & $-1$& 0& 0;& 0& 1& 1& 0;& 1& 0& 0& $-1;$& 0& 0& 1& $\left. 0\right]$ & $\left(2, 0, 2\right)$\\
& & $[$& -1 & $-1$& 0& 0;& 0& 0& 1& 0;& 1& 0& 0& 1;& 0& 0& 1& $\left. 0\right]$ & $\left(1, 1, 2\right)$\\
& & $[$& 0 & 0& 0& 0;& 0& 1& 1& 0;& 1& 0& 0& $-1;$& 0& 0& 1& $\left. 0\right]$ & $\left(1, 0, 3\right)$\\
& & $[$& 0 & 0& 0& 0;& 0& 0& 1& 0;& 1& 0& 0& $-1;$& 0& 0& 1& $\left. 0\right]$ & $\left(0, 0, 4\right)$\\
\hline
\end{longtable}
\begin{longtable}{rrlrrrrrrrrrrrrrrrrr}
\caption{Inertially arbitrary patterns with no proper $2$-cycle.}\label{IAP4LastCase}\\
\hline
\hline
\multicolumn{2}{c}{Pattern}  & \multicolumn{17}{c}{Matrix} & Inertia \\
\hline
\hline

\multirow{5}{*}{$\J_1$}
&   \multirow{5}{*}{\tikzpicture \phantom{\node (5) at (-0.25,0.28)[place]{};}
\node (1) at (-0.5,0.5)[place] {};
\node (2) at (0.5,0.5)[place] {};
\node (3) at (0.5,-0.5)[place] {};
\node (4) at (-0.5,-0.5)[place] {};
\draw [right] (3) to (4);
\draw [right] (4) to (1);
\draw [right] (1) to (2);
\draw [right] (2) to (3);
\draw [right] (1) to (3);
\draw [-] (-0.5,0.55) arc (360:0:3pt);
\draw [-] (0.7,0.55) arc (360:0:3pt);
\endtikzpicture}
  & $[$& 2& 7& 4& 0;& 0& 2& 1& 0;& 0& 0& 0& 1;& 1& 0& 0;& $\left. 0\right]$ & $\left(4, 0, 0\right)$\\
& & $[$& $-1$& 0& $-1$& 0;& 0& 1& 1& 0;& 0& 0& 0& 1;& 1& 0& 0;& $\left. 0\right]$ & $\left(3, 1, 0\right)$\\
& & $[$& 1& 1& 1& 0;& 0& 1& 1& 0;& 0& 0& 0& 1;& 1& 0& 0;& $\left. 0\right]$ & $\left(3, 0, 1\right)$\\
& & $[$& $-1$& $-1$& $-1$& 0;& 0& $-1$& 1& 0;& 0& 0& 0& 1;& 1& 0& 0;& $\left. 0\right]$ & $\left(2, 2, 0\right)$\\
& & $[$& $-1$& $-1$& $-1$& 0;& 0& 1& 1& 0;& 0& 0& 0& 1;& 1& 0& 0;& $\left. 0\right]$ & $\left(2, 1, 1\right)$\\
& & $[$& 1& 0& 0& 0;& 0& 1& 1& 0;& 0& 0& 0& 1;& 1& 0& 0;& $\left. 0\right]$ & $\left(2, 0, 2\right)$\\
& & $[$& $-1$& 0& 0& 0;& 0& 1& 1& 0;& 0& 0& 0& 1;& 1& 0& 0;& $\left. 0\right]$ & $\left(1, 1, 2\right)$\\
& & $[$& 0& 0& 0& 0;& 0& 1& 1& 0;& 0& 0& 0& 1;& 1& 0& 0;& $\left. 0\right]$ & $\left(1, 0, 3\right)$\\
& & $[$& 0& 0& 0& 0;& 0& 0& 1& 0;& 0& 0& 0& 1;& 1& 0& 0;& $\left. 0\right]$ & $\left(0, 0, 4\right)$\\
\hline

\multirow{5}{*}{$\J_2$}
&   \multirow{5}{*}{\tikzpicture \phantom{\node (5) at (-0.25,0.28)[place]{};}
\node (1) at (-0.5,0.5)[place] {};
\node (2) at (0.5,0.5)[place] {};
\node (3) at (0.5,-0.5)[place] {};
\node (4) at (-0.5,-0.5)[place] {};
\draw [right] (3) to (4);
\draw [right] (4) to (1);
\draw [right] (1) to (2);
\draw [right] (2) to (3);
\draw [right] (3) to (1);
\draw [-] (-0.5,0.55) arc (360:0:3pt);
\draw [-] (0.7,0.55) arc (360:0:3pt);
\endtikzpicture}
  & $[$& 1& 1& 0& 0;& 0& 3& 1& 0;& 4& 0& 0& $-1$;& 1& 0& 0;& $\left. 0\right]$ & $\left(4, 0, 0\right)$\\
& & $[$& $-1$& 1& 0& 0;& 0& 1& 1& 0;& $-1$& 0& 0& 1;& 1& 0& 0;& $\left. 0\right]$ & $\left(3, 1, 0\right)$\\
& & $[$& 1& 1& 0& 0;& 0& 1& 1& 0;& 1& 0& 0& 0;& 1& 0& 0;& $\left. 0\right]$ & $\left(3, 0, 1\right)$\\
& & $[$& $-1$& 1& 0& 0;& 0& $-1$& 1& 0;& $-1$& 0& 0& $-1$;& 1& 0& 0;& $\left. 0\right]$ & $\left(2, 2, 0\right)$\\
& & $[$& $-1$& 1& 0& 0;& 0& 0& 1& 0;& $-1$& 0& 0& 0;& 1& 0& 0;& $\left. 0\right]$ & $\left(2, 1, 1\right)$\\
& & $[$& 1& 1& 0& 0;& 0& 1& 1& 0;& 0& 0& 0& 0;& 1& 0& 0;& $\left. 0\right]$ & $\left(2, 0, 2\right)$\\
& & $[$& $-1$& 1& 0& 0;& 0& 0& 1& 0;& $-1$& 0& 0& 1;& 1& 0& 0;& $\left. 0\right]$ & $\left(1, 1, 2\right)$\\
& & $[$& 0& 1& 0& 0;& 0& 1& 1& 0;& 0& 0& 0& 0;& 1& 0& 0;& $\left. 0\right]$ & $\left(1, 0, 3\right)$\\
& & $[$& 0& 1& 0& 0;& 0& 0& 1& 0;& 0& 0& 0& 0;& 1& 0& 0;& $\left. 0\right]$ & $\left(0, 0, 4\right)$\\
\hline

\multirow{5}{*}{$\J_3$}
&   \multirow{5}{*}{\tikzpicture \phantom{\node (5) at (-0.25,0.28)[place]{};}
\node (1) at (-0.5,0.5)[place] {};
\node (2) at (0.5,0.5)[place] {};
\node (3) at (0.5,-0.5)[place] {};
\node (4) at (-0.5,-0.5)[place] {};
\draw [right] (3) to (4);
\draw [right] (4) to (1);
\draw [right] (1) to (2);
\draw [right] (2) to (3);
\draw [right] (1) to (3);
\draw [-] (-0.5,0.55) arc (360:0:3pt);
\draw [-] (0.7,-0.55) arc (360:0:3pt);
\endtikzpicture}
  & $[$& 3& $-1$& 4& 0;& 0& 0& 1& 0;& 0& 0& 1& 1;& 1& 0& 0;& $\left. 0\right]$ & $\left(4, 0, 0\right)$\\
& & $[$& $-1$& 1& $-1$& 0;& 0& 0& 1& 0;& 0& 0& 1& 1;& 1& 0& 0;& $\left. 0\right]$ & $\left(3, 1, 0\right)$\\
& & $[$& 1& 0& 1& 0;& 0& 0& 1& 0;& 0& 0& 1& 1;& 1& 0& 0;& $\left. 0\right]$ & $\left(3, 0, 1\right)$\\
& & $[$& $-1$& $-1$& $-1$& 0;& 0& 0& 1& 0;& 0& 0& $-1$& 1;& 1& 0& 0;& $\left. 0\right]$ & $\left(2, 2, 0\right)$\\
& & $[$& $-1$& 0& $-1$& 0;& 0& 0& 1& 0;& 0& 0& 0& 1;& 1& 0& 0;& $\left. 0\right]$ & $\left(2, 1, 1\right)$\\
& & $[$& 1& 0& 0& 0;& 0& 0& 1& 0;& 0& 0& 1& 1;& 1& 0& 0;& $\left. 0\right]$ & $\left(2, 0, 2\right)$\\
& & $[$& $-1$& 0& 0& 0;& 0& 0& 1& 0;& 0& 0& 1& 1;& 1& 0& 0;& $\left. 0\right]$ & $\left(1, 1, 2\right)$\\
& & $[$& 0& 0& 0& 0;& 0& 0& 1& 0;& 0& 0& 1& 1;& 1& 0& 0;& $\left. 0\right]$ & $\left(1, 0, 3\right)$\\
& & $[$& 0& 0& 0& 0;& 0& 0& 1& 0;& 0& 0& 0& 1;& 1& 0& 0;& $\left. 0\right]$ & $\left(0, 0, 4\right)$\\
\hline

\multirow{5}{*}{$\J_4$}
&   \multirow{5}{*}{\tikzpicture \phantom{\node (5) at (-0.25,0.28)[place]{};}
\node (1) at (-0.5,0.5)[place] {};
\node (2) at (0.5,0.5)[place] {};
\node (3) at (0.5,-0.5)[place] {};
\node (4) at (-0.5,-0.5)[place] {};
\draw [right] (3) to (4);
\draw [right] (4) to (1);
\draw [right] (1) to (2);
\draw [right] (2) to (3);
\draw [right] (2) to (4);
\draw [-] (-0.5,0.55) arc (360:0:3pt);
\draw [-] (0.7,-0.55) arc (360:0:3pt);
\endtikzpicture}
  & $[$& 1& 1& 0& 0;& 0& 0& 11& 4;& 0& 0& 3& 1;& 1& 0& 0;& $\left. 0\right]$ & $\left(4, 0, 0\right)$\\
& & $[$& $-1$& 1& 0& 0;& 0& 0& 0& $-1$;& 0& 0& 1& 1;& 1& 0& 0;& $\left. 0\right]$ & $\left(3, 1, 0\right)$\\
& & $[$& 1& 1& 0& 0;& 0& 0& 1& 1;& 0& 0& 1& 1;& 1& 0& 0;& $\left. 0\right]$ & $\left(3, 0, 1\right)$\\
& & $[$& $-1$& 1& 0& 0;& 0& 0& $-1$& $-1$;& 0& 0& $-1$& 1;& 1& 0& 0;& $\left. 0\right]$ & $\left(2, 2, 0\right)$\\
& & $[$& $-1$& 1& 0& 0;& 0& 0& $-1$& $-1$;& 0& 0& 1& 1;& 1& 0& 0;& $\left. 0\right]$ & $\left(2, 1, 1\right)$\\
& & $[$& 1& 1& 0& 0;& 0& 0& 0& 0;& 0& 0& 1& 1;& 1& 0& 0;& $\left. 0\right]$ & $\left(2, 0, 2\right)$\\
& & $[$& $-1$& 1& 0& 0;& 0& 0& 0& 0;& 0& 0& 1& 1;& 1& 0& 0;& $\left. 0\right]$ & $\left(1, 1, 2\right)$\\
& & $[$& 0& 1& 0& 0;& 0& 0& 0& 0;& 0& 0& 1& 1;& 1& 0& 0;& $\left. 0\right]$ & $\left(1, 0, 3\right)$\\
& & $[$& 0& 1& 0& 0;& 0& 0& 0& 0;& 0& 0& 0& 1;& 1& 0& 0;& $\left. 0\right]$ & $\left(0, 0, 4\right)$\\
\hline

\multirow{5}{*}{$\J_5$}
&   \multirow{5}{*}{\tikzpicture \phantom{\node (5) at (-0.25,0.28)[place]{};}
\node (1) at (-0.5,0.5)[place] {};
\node (2) at (0.5,0.5)[place] {};
\node (3) at (0.5,-0.5)[place] {};
\node (4) at (-0.5,-0.5)[place] {};
\draw [right] (4) to (3);
\draw [right] (4) to (1);
\draw [right] (1) to (2);
\draw [right] (3) to (2);
\draw [right] (2) to (4);
\draw [-] (-0.5,0.55) arc (360:0:3pt);
\draw [-] (0.7,0.55) arc (360:0:3pt);
\endtikzpicture}
  & $[$& 1& 3& 0& 0;& 0& 3& 0& 1;& 0& 1& 0& 0;& 1& 0& 1;& $\left. 0\right]$ & $\left(4, 0, 0\right)$\\
& & $[$& 1& $-1$& 0& 0;& 0& $-1$& 0& 1;& 0& $-1$& 0& 0;& 1& 0& 1;& $\left. 0\right]$ & $\left(3, 1, 0\right)$\\
& & $[$& 1& 1& 0& 0;& 0& 1& 0& 1;& 0& 0& 0& 0;& 1& 0& 1;& $\left. 0\right]$ & $\left(3, 0, 1\right)$\\
& & $[$& $-1$& $-1$& 0& 0;& 0& $-1$& 0& 1;& 0& $-1$& 0& 0;& 1& 0& 1;& $\left. 0\right]$ & $\left(2, 2, 0\right)$\\
& & $[$& $-1$& $-1$& 0& 0;& 0& 0& 0& 1;& 0& 0& 0& 0;& 1& 0& 1;& $\left. 0\right]$ & $\left(2, 1, 1\right)$\\
& & $[$& 1& 0& 0& 0;& 0& 1& 0& 1;& 0& 0& 0& 0;& 1& 0& 1;& $\left. 0\right]$ & $\left(2, 0, 2\right)$\\
& & $[$& $-1$& $-1$& 0& 0;& 0& 1& 0& 1;& 0& 1& 0& 0;& 1& 0& 1;& $\left. 0\right]$ & $\left(1, 1, 2\right)$\\
& & $[$& 0& $-1$& 0& 0;& 0& 1& 0& 1;& 0& 1& 0& 0;& 1& 0& 1;& $\left. 0\right]$ & $\left(1, 0, 3\right)$\\
& & $[$& 0& $-1$& 0& 0;& 0& 0& 0& 1;& 0& 1& 0& 0;& 1& 0& 1;& $\left. 0\right]$ & $\left(0, 0, 4\right)$\\
\hline

\multirow{5}{*}{$\J_6$}
&   \multirow{5}{*}{\tikzpicture \phantom{\node (5) at (-0.25,0.28)[place]{};}
\node (1) at (-0.5,0.5)[place] {};
\node (2) at (0.5,0.5)[place] {};
\node (3) at (0.5,-0.5)[place] {};
\node (4) at (-0.5,-0.5)[place] {};
\draw [right] (4) to (3);
\draw [right] (4) to (1);
\draw [right] (1) to (2);
\draw [right] (3) to (2);
\draw [right] (2) to (4);
\draw [-] (-0.5,0.55) arc (360:0:3pt);
\draw [-] (0.7,-0.55) arc (360:0:3pt);
\endtikzpicture}
  & $[$& 1& $-3/2$& 0& 0;& 0& 0& 0& 1;& 0& 11/2& 3& 0;& 1& 0& 1;& $\left. 0\right]$ & $\left(4, 0, 0\right)$\\
& & $[$& $-1$& $-1$& 0& 0;& 0& 0& 0& 1;& 0& 0& 1& 0;& 1& 0& 1;& $\left. 0\right]$ & $\left(3, 1, 0\right)$\\
& & $[$& 1& $-1$& 0& 0;& 0& 0& 0& 1;& 0& 2& 2& 0;& 1& 0& 1;& $\left. 0\right]$ & $\left(3, 0, 1\right)$\\
& & $[$& $-1$& $-1$& 0& 0;& 0& 0& 0& 1;& 0& $-1$& $-1$& 0;& 1& 0& 1;& $\left. 0\right]$ & $\left(2, 2, 0\right)$\\
& & $[$& $-1$& $-1$& 0& 0;& 0& 0& 0& 1;& 0& $-1$& 1& 0;& 1& 0& 1;& $\left. 0\right]$ & $\left(2, 1, 1\right)$\\
& & $[$& 1& $-1$& 0& 0;& 0& 0& 0& 1;& 0& 1& 1& 0;& 1& 0& 1;& $\left. 0\right]$ & $\left(2, 0, 2\right)$\\
& & $[$& $-1$& $-1$& 0& 0;& 0& 0& 0& 1;& 0& 1& 1& 0;& 1& 0& 1;& $\left. 0\right]$ & $\left(1, 1, 2\right)$\\
& & $[$& 0& 0& 0& 0;& 0& 0& 0& 1;& 0& 0& 1& 0;& 1& 0& 1;& $\left. 0\right]$ & $\left(1, 0, 3\right)$\\
& & $[$& 0& $-1$& 0& 0;& 0& 0& 0& 1;& 0& 1& 0& 0;& 1& 0& 1;& $\left. 0\right]$ & $\left(0, 0, 4\right)$\\
\hline
\end{longtable}

\end{document}